\documentclass[preprint,review,10pt]{elsarticle}

\usepackage{amssymb}
\usepackage{amssymb,amsmath,amsthm,epsfig,multicol}
\usepackage{mathtools} 
\usepackage{epsfig}
\usepackage{fullpage}
\usepackage{subfig}
\usepackage{listings}
\usepackage{float}
\usepackage{epstopdf}
\usepackage{color}
\definecolor{mygreen}{RGB}{28,172,0}
\usepackage{xcolor}
\usepackage{lineno}
\newtheorem{thm}{Theorem}[section]
\newtheorem{lem}[thm]{Lemma}
\newtheorem{cor}[thm]{Corollary}
\newtheorem{prop}[thm]{Proposition}
\theoremstyle{definition}
\newtheorem{defn}[thm]{Definition}

\newtheorem{rem}[thm]{Remark}
\theoremstyle{remark} \numberwithin{equation}{section}

\newcommand{\R}{\mathbb{R}}

\DeclareMathOperator*{\esssup}{ess\,sup}

\definecolor{hmfcol}{rgb}{0,0,0.8}
\definecolor{afcol}{rgb}{1,0,0}

\biboptions{sort&compress,}

 \journal{}

\begin{document}

\begin{frontmatter}



 \title{Tempered and Hadamard-type fractional calculus with respect to functions}

\author[a,b]{ Hafiz Muhammad Fahad}
\ead{hafizmuhammadfahad13@gmail.com} 

\author[b]{\corref{cor2}  Arran Fernandez}
\ead{arran.fernandez@emu.edu.tr} \cortext[cor2]{Corresponding
	author.}

\author[a]{ Mujeeb ur Rehman}
\ead{mujeeburrehman345@yahoo.com}

\author[a,c]{\corref{cor1}  Maham Siddiqi}
\ead{maham26@gmail.com} \cortext[cor1]{She was a visiting researcher at the School of Natural Sciences, NUST, Islamabad, Pakistan.}

 \address[a]{Department of Mathematics, School of Natural Sciences,
 National University of Sciences and Technology,\\
    Islamabad, Pakistan}
 \address[b]{Department of Mathematics, Faculty of Arts and Sciences, Eastern Mediterranean University,\\ Famagusta, Northern Cyprus, via Mersin 10, Turkey}
 \address[c]{Department of Space Science, Institute of Space Technology, Islamabad, Pakistan}

  \begin{abstract}
  Many different types of fractional calculus have been defined, which may be categorised into broad classes according to their properties and behaviours. Two types that have been much studied in the literature are the Hadamard-type fractional calculus and tempered fractional calculus. This paper establishes a connection between these two definitions, writing one in terms of the other by making use of the theory of fractional calculus with respect to functions. By extending this connection in a natural way, a generalisation is developed which unifies several existing fractional operators: Riemann--Liouville, Caputo, classical Hadamard, Hadamard-type, tempered, and all of these taken with respect to functions. The fundamental calculus of these generalised operators is established, including semigroup and reciprocal properties as well as application to some example functions. Function spaces are constructed in which the new operators are defined and bounded. Finally, some formulae are derived for fractional integration by parts with these operators.
\end{abstract}


\end{frontmatter}

\section{Introduction}

Fractional calculus is a field of interest for scientists because of its many applications in diverse areas of study, including physics, engineering, biology, and economics. Its history dates back to the end of the seventeenth century, when Leibniz established the symbol $\frac{\mathrm{d}^{n}}{\mathrm{d}x^{n}} f(x)$ for the $n$th derivative of a function $f$, and de l'H\^opital raised the question of what this could mean when $n$ is a fraction such as $\frac{1}{2}$. Fractional differentiation and fractional integration are generalizations of the ideas of integer-order differentiation and integration, and include $n$-th derivatives and $n$-fold integrals as special cases. We refer to \cite{Samko,Kilbas,Podlubny} for the mathematical theory and \cite{Hilfer,Podlubny2,Diethelm,Herzallah,Tarasov} for the applications in science.

There has been a lot of literature and research especially on the classical Riemann--Liouville and Caputo fractional calculus. The Riemann--Liouville fractional derivative is the most well-established, and in some senses the most natural, definition, but it has some disadvantages when used in modelling physical problems, because the required initial value conditions are themselves fractional. The Caputo fractional derivative is often more suitable for physical conditions because it requires only initial conditions in the classical form \cite{Diethelm}. 

Several different fractional operators have been defined, out of which Riemann--Liouville, Caputo, Hilfer, Riesz, Erdelyi--Kober, Hadamard are just a few to mention \cite{Samko,Kilbas}. Each definition has its own conditions and properties, and many of them are not equivalent to each other. In practice, the physical system under consideration determines the selection of a suitable fractional operator, and the numerous inequivalent definitions are each useful in their own contexts. Therefore, it is logical that we should investigate and develop fractional operators that are generalised versions of the existing, specific cases. Certain very general fractional operators may be considered as classes, which include within them several different fractional operators of practical value. This is a more efficient way of developing the mathematical theory, compared with deriving the same proofs multiple times in many different, but similar, models of fractional calculus \cite{baleanu-fernandez}.

One such class is the class of fractional operators with analytic kernels, proposed in \cite{fernandez-ozarslan-baleanu} as a way of generalising many types of fractional calculus which have been intensely studied in the last few years. One example within this class will be a major topic of this paper, namely \textit{tempered} fractional calculus \cite{meerschaert-sabzikar-chen}, which has also been called \textit{substantial} fractional calculus \cite{cao-li-chen} and \textit{generalised proportional} fractional calculus \cite{jarad-abdeljawad-alzabut}.

Another class of fractional operators, sometimes collectively called $\Psi$-fractional calculus, is given by fractional integration and differentiation of a function with respect to another function: the fractional equivalent of the Riemann--Stieltjes integral. The idea first arose in a 1964 paper \cite{erdelyi} in which Erdelyi discussed fractional integration with respect to a power function, defining an operator which Katugampola rediscovered in 2011 \cite{Katugampola} and which nowadays is sometimes called Katugampola fractional calculus. In its full generality, this class was first proposed and motivated by Osler in 1970 \cite{Osler}, and it was mentioned briefly in the 1974 textbook \cite{oldham-spanier}, although it is often cited instead to the 2006 textbook \cite{Kilbas} where it is analysed in more detail. One example within this class is the \textit{Hadamard} fractional calculus, dating back to 1892 \cite{hadamard}, in which integrals and derivatives are taken with respect to a logarithm function. A variant of this definition, called \textit{Hadamard-type} fractional calculus, was introduced in \cite{Butzer} and further analysed in other works such as \cite{3} and the textbook \cite{Kilbas}, with recent advances including extensions of the operators to Banach spaces \cite{cichon-salem,salem}.

In this paper, we establish an important connection between tempered and Hadamard-type fractional calculus, which, to the best of our knowledge, has not yet been noticed until the present day. Like all new connections in mathematics, this is expected to be useful in the understanding and analysis of both topics. Following the philosophy of defining general classes for mathematical research, we then examine a generalisation of both tempered and Hadamard-type fractional operators, a class of operators which includes both of these standard types, by using the theory of fractional calculus with respect to functions.

The structure of this paper is as follows. Section \ref{Sec:prelim} provides the necessary background, with definitions of the fractional operators mentioned above and statements of a few fundamental facts concerning them. Section \ref{Sec:main} contains the key new ideas: in section 3.1, the connection between tempered and Hadamard-type fractional calculus, and in section 3.2, the generalised form which we spend the rest of the paper studying. Section \ref{Sec:fnspace} concerns function spaces for the generalised operators: for the fractional integral operator in section 4.1, and for the fractional derivative operator in section 4.2. Section \ref{Sec:intparts} establishes integration by parts formulae in the setting of the tempered or Hadamard-type fractional operators with respect to functions. Section \ref{Sec:concl} finally makes concluding statements about our main results and potential future research in this area.

\section{Definitions and background} \label{Sec:prelim}

In this section, we provide the definitions for various types of fractional calculus which we shall refer to throughout the paper. We also recall some essential results whose proofs can be seen in the literature.

\begin{defn}[\cite{Samko,oldham-spanier,Caputo}]
\label{Def:RL&C}
The Riemann--Liouville fractional integral with order $\mu>0$ (or $\mu\in\mathbb{C}$ with $\mathrm{Re}(\mu)>0$) of a given function $f$ is defined by
\[
\prescript{}{a}I^{\mu}_xf(x)=\frac{1}{\Gamma(\mu)}\int_a^x(x-t)^{\mu-1}f(t)\,\mathrm{d}t,
\]
where $x\in(a,b)$ and $a<b$ in $\mathbb{R}$. This is the fractional power of the standard differentiation operator $\frac{\mathrm{d}}{\mathrm{d}x}$.

The Riemann--Liouville fractional derivative with order $\mu>0$ (or $\mu\in\mathbb{C}$ with $\mathrm{Re}(\mu)\geq0$) of a given function $f$ is defined by
\[
\prescript{R}{a}D^{\mu}_xf(x)=\left(\frac{\mathrm{d}}{\mathrm{d}x}\right)^n\prescript{}{a}I^{n-\mu}_xf(x),
\]
where $n-1\leq\mathrm{Re}(\mu)<n\in\mathbb{Z}^+$, $x\in(a,b)$, and $a<b$ in $\mathbb{R}$.

The Caputo fractional derivative with order $\mu>0$ (or $\mu\in\mathbb{C}$ with $\mathrm{Re}(\mu)\geq0$) of a given function $f$ is defined by
\[
\prescript{C}{a}D^{\mu}_xf(x)=\prescript{}{a}I^{n-\mu}_x\left(\frac{\mathrm{d}}{\mathrm{d}x}\right)^nf(x),
\]
where $n-1\leq\mathrm{Re}(\mu)<n\in\mathbb{Z}^+$, $x\in(a,b)$, and $a<b$ in $\mathbb{R}$.
\end{defn}

We note that the Riemann--Liouville and Caputo fractional derivatives both stem from the same definition of fractional integrals, simply combining this with the original differentiation operation in one order or the other. 

\begin{defn}[\cite{meerschaert-sabzikar-chen,li-deng-zhao}]
\label{Def:tempered}
The tempered fractional integral with order $\mu>0$ (or $\mu\in\mathbb{C}$ with $\mathrm{Re}(\mu)>0$) and parameter $s\in\mathbb{C}$ of a given function $f$ is defined by
\[
\prescript{T}{a}I^{\mu,s}_xf(x)=\frac{1}{\Gamma(\mu)}\int_a^x(x-t)^{\mu-1}e^{-s(x-t)}f(t)\,\mathrm{d}t=e^{-sx}\prescript{}{a}I^{\mu}_x\big(e^{sx}f(x)\big),
\]
where $x\in(a,b)$ and $a<b$ in $\mathbb{R}$. This is the fractional power of the modified differentiation operator $\left(\frac{\mathrm{d}}{\mathrm{d}x}+s\right)$.

The tempered fractional derivatives of Riemann--Liouville and Caputo type, with order $\mu>0$ (or $\mu\in\mathbb{C}$ with $\mathrm{Re}(\mu)\geq0$) and parameter $s\in\mathbb{C}$, of a given function $f$, are defined respectively by:
\begin{align*}
\prescript{TR}{a}D^{\mu,s}_xf(x)&=\left(\frac{\mathrm{d}}{\mathrm{d}x}+s\right)^n\prescript{T}{a}I^{n-\mu}_xf(x)=e^{-sx}\prescript{R}{a}D^{\mu}_x\big(e^{sx}f(x)\big), \\
\prescript{TC}{a}D^{\mu,s}_xf(x)&=\prescript{T}{a}I^{n-\mu}_x\left(\frac{\mathrm{d}}{\mathrm{d}x}+s\right)^nf(x)=e^{-sx}\prescript{C}{a}D^{\mu}_x\big(e^{sx}f(x)\big),
\end{align*}
where $n-1\leq\mathrm{Re}(\mu)<n\in\mathbb{Z}^+$, $x\in(a,b)$, and $a<b$ in $\mathbb{R}$.
\end{defn}

The operators of tempered fractional calculus belong to the general class of fractional operators defined by integrals with analytic kernel functions \cite{fernandez-ozarslan-baleanu,baleanu-fernandez}: in this case, the kernel is a fractional power function multiplied by an exponential function. Interestingly, these operators can also be seen as the conjugation of the corresponding original operators from Definition \ref{Def:RL&C} with the simple operation of multiplication by an exponential function. This is equivalent to taking fractional powers of $\left(\frac{\mathrm{d}}{\mathrm{d}x}+s\right)$ because we have precisely
\[
\left(\frac{\mathrm{d}}{\mathrm{d}x}+s\right)f(x)=e^{-sx}\frac{\mathrm{d}}{\mathrm{d}x}\big(e^{sx}f(x)\big).
\]
The same definition has been proposed in various contexts. In \cite{friedrich-jenko}, a physical study of anomalous dynamics of particles gave rise to so-called \textit{substantial} fractional operators, which were compared with tempered fractional operators in \cite{cao-li-chen}. In \cite{jarad-abdeljawad-alzabut}, a mathematical study of fractional iteration gave rise to so-called \textit{generalised proportional} fractional operators, which were compared with tempered fractional operators in \cite{fernandez-ustaoglu}.

\begin{defn}[\cite{Osler,Kilbas,Almeida}]
\label{Def:FwrtF}
The Riemann--Liouville fractional integral with order $\mu>0$ (or $\mu\in\mathbb{C}$ with $\mathrm{Re}(\mu)>0$) of a given function $f$ with respect to a monotonic $C^1$ function $g$ is defined as
\[
\prescript{}{a}I^{\mu}_{g(x)}f(x)=\frac{1}{\Gamma(\mu)}\int_a^x\big(g(x)-g(t)\big)^{\mu-1}f(t)g'(t)\,\mathrm{d}t,
\]
where $x\in(a,b)$ and $a<b$ in $\mathbb{R}$. This is the fractional power of the operator $\frac{\mathrm{d}}{\mathrm{d}g(x)}=\frac{1}{g'(x)}\cdot\frac{\mathrm{d}}{\mathrm{d}x}$ of differentiation with respect to the function $g$.

The Riemann--Liouville and Caputo fractional derivatives with order $\mu>0$ (or $\mu\in\mathbb{C}$ with $\mathrm{Re}(\mu)\geq0$) of a given function $f$, with respect to a monotonic $C^1$ function $g$, are defined respectively as
\begin{align*}
\prescript{R}{a}D^{\mu}_{g(x)}f(x)&=\left(\frac{1}{g'(x)}\cdot\frac{\mathrm{d}}{\mathrm{d}x}\right)^n\prescript{}{a}I^{n-\mu}_{g(x)}f(x), \\
\prescript{C}{a}D^{\mu}_{g(x)}f(x)&=\prescript{}{a}I^{n-\mu}_{g(x)}\left(\frac{1}{g'(x)}\cdot\frac{\mathrm{d}}{\mathrm{d}x}\right)^nf(x),
\end{align*}
where $n-1\leq\mathrm{Re}(\mu)<n\in\mathbb{Z}^+$, $x\in(a,b)$, and $a<b$ in $\mathbb{R}$.
\end{defn}

Just like in Definition \ref{Def:tempered}, the fractional operators defined in Definition \ref{Def:FwrtF} can be written as the conjugation of the original fractional operators of Definition \ref{Def:RL&C} with a simpler operation, this time the operation of composition with $g$ or $g^{-1}$:
\begin{equation}
\label{FwrtF:conjug}
\prescript{}{a}I^{\mu}_{g(x)}=Q_g\circ\prescript{}{g(a)}I^{\mu}_x\circ Q_g^{-1},\quad\quad\prescript{R}{a}D^{\mu}_{g(x)}=Q_g\circ\prescript{R}{g(a)}D^{\mu}_x\circ Q_g^{-1},\quad\quad\prescript{C}{a}D^{\mu}_{g(x)}=Q_g\circ\prescript{C}{g(a)}D^{\mu}_x\circ Q_g^{-1},
\end{equation}
where the functional operator $Q_g$ is defined by
\[
(Q_gf)(x)=f(g(x)).
\]

One case which we wish to focus on is the so-called Hadamard fractional calculus, in which we use Definition \ref{Def:FwrtF} with $g(x)=\log(x)$, the natural logarithm function. Thus the Hadamard fractional integral is
\begin{equation}
\label{Hadamard}
\prescript{H}{a}I^{\mu}_xf(x)=\prescript{}{a}I^{\mu}_{\log x}f(x)=\frac{1}{\Gamma(\mu)}\int_a^x\left(\log\frac{x}{t}\right)^{\mu-1}\frac{f(t)}{t}\,\mathrm{d}t,
\end{equation}
which is the fractional power of the operator $x\frac{\mathrm{d}}{\mathrm{d}x}$, and the Hadamard fractional derivatives of Riemann--Liouville and Caputo type \cite{hadamard,AbdeljawadT} are given as usual by composition of this fractional integral with positive integer powers of $x\frac{\mathrm{d}}{\mathrm{d}x}$.

Hadamard fractional calculus is just a special case of fractional calculus with respect to a function. The more general Hadamard-type fractional calculus, however, is not, and we give it its own Definition as follows.

\begin{defn}[\cite{3,Butzer}]
\label{Def:HT}
The Hadamard-type fractional integral with order $\mu>0$ (or $\mu\in\mathbb{C}$ with $\mathrm{Re}(\mu)>0$) and parameter $s\in\mathbb{C}$ of a given function $f$ is defined as
\[
\prescript{H}{a}I^{\mu,s}_xf(x)=\frac{1}{\Gamma(\mu)}\int_a^x\left(\frac{t}{x}\right)^s\left(\log\frac{x}{t}\right)^{\mu-1}\frac{f(t)}{t}\,\mathrm{d}t,
\]
where $x\in(a,b)$ and $a<b$ in $\mathbb{R}$.

The Hadamard-type fractional derivatives of Riemann--Liouville and Caputo type, with order $\mu>0$ (or $\mu\in\mathbb{C}$ with $\mathrm{Re}(\mu)\geq0$) and parameter $s\in\mathbb{C}$, of a given function $f$, are defined respectively as
\begin{align*}
\prescript{HR}{a}D^{\mu,s}_xf(x)&=x^{-s}\left(x\frac{\mathrm{d}}{\mathrm{d}x}\right)^n\Big(x^s\prescript{H}{a}I^{n-\mu}_xf(x)\Big), \\
\prescript{HC}{a}D^{\mu,s}_xf(x)&=\prescript{T}{a}I^{n-\mu}_x\left(x^{-s}\left(x\frac{\mathrm{d}}{\mathrm{d}x}\right)^n\big(x^sf(x)\big)\right),
\end{align*}
where $n-1\leq\mathrm{Re}(\mu)<n\in\mathbb{Z}^+$, $x\in(a,b)$, and $a<b$ in $\mathbb{R}$.
\end{defn}

For the special case $s = 0$, the Hadamard-type fractional integral and derivative, as given in Definition \ref{Def:HT}, reduce to precisely the classical Hadamard fractional integral and derivative.

Finally, we state some basic results about fractional integrals and derivatives with respect to functions. These are easily deduced from the corresponding results on the original Riemann--Liouville operators by using the relations \eqref{FwrtF:conjug}, as discussed in \cite{Samko,Kilbas}.

\begin{prop}
\label{Prop:FwrtF1}
The expressions $\prescript{}{a}I^{\mu}_{g(x)}f(x)$ and $\prescript{R}{a}D^{\mu}_{g(x)}f(x)$ each define continuous (indeed, analytic) functions of $\mu$ on the respective half-lines (or half-planes in $\mathbb{C}$). If $f$ is continuous and differentiable so that both are defined, then they meet continuously at $\mu=0$, so that we can write $\prescript{}{a}I^{\mu}_{g(x)}f(x)=\prescript{R}{a}D^{-\mu}_{g(x)}f(x)$. In particular:
\begin{itemize}
	\item[(a)] $ \displaystyle\lim_{\mu \to 0^+}\prescript{}{a}I^{\mu}_{g(x)} f(x) = f(x) = \lim_{\mu \to 0^+}\prescript{R}{a}D^{\mu}_{g(x)} f(x)$;
	\item[(b)] $ \displaystyle\lim_{\mu \to (n-1)^+}\prescript{R}{a}D^{\mu}_{g(x)} f(x) = \left(\frac{1}{g'(x)}\cdot\frac{\mathrm{d}}{\mathrm{d}x}\right)^{n-1} f(x), \quad\quad \lim_{\mu \to n^-}\prescript{R}{a}D^{\mu}_{g(x)} f(x) = \left(\frac{1}{g'(x)}\cdot\frac{\mathrm{d}}{\mathrm{d}x}\right)^n f(x)$.
\end{itemize}
\end{prop}

\begin{prop}
\label{Prop:FwrtF2}
The fractional integrals and derivatives with respect to $g(x)$ have the following semigroup properties:
\begin{align*}
\prescript{}{a}I^{\mu}_{g(x)}f(x)\prescript{}{a}I^{\nu}_{g(x)}f(x)&=\prescript{}{a}I^{\mu+\nu}_{g(x)}f(x),\quad\quad\mu,\nu>0; \\
\prescript{R}{a}D^{\mu}_{g(x)}f(x)\prescript{}{a}I^{\nu}_{g(x)}f(x)&=\prescript{R}{a}D^{\mu-\nu}_{g(x)}f(x),\quad\quad\nu>0; \\
\left(\frac{1}{g'(x)}\cdot\frac{\mathrm{d}}{\mathrm{d}x}\right)^n\prescript{R}{a}D^{\mu}_{g(x)}f(x)&=\prescript{R}{a}D^{n+\mu}_{g(x)}f(x),
\end{align*}
where in the last two identities the operators denoted by $D$ may be either fractional derivatives or fractional integrals according to the sign of the index.
\end{prop}

\begin{prop}
\label{Prop:FwrtF3}
For $\nu>0$ and any $\mu$, the following relations hold:
\begin{align*}
\prescript{}{a}I^{\mu}_{g(x)}\Big(\big(g(x)-g(a)\big)^{\nu-1}\Big)&=\frac{\Gamma(\nu)}{\Gamma(\nu+\mu)}\big(g(x)-g(a)\big)^{\nu+\mu-1}; \\
\prescript{R}{a}D^{\mu}_{g(x)}\Big(\big(g(x)-g(a)\big)^{\nu-1}\Big)&=\frac{\Gamma(\nu)}{\Gamma(\nu-\mu)}\big(g(x)-g(a)\big)^{\nu-\mu-1}.
\end{align*}
\end{prop}

\section{Unifying tempered and Hadamard-type fractional calculus} \label{Sec:main}

In this section, we prove a new relationship between tempered fractional calculus and Hadamard-type fractional calculus. Then, using the concept of fractional integration and differentiation with respect to functions, we show how both of these, apparently different, types of fractional calculus may be unified into a single general system.

\subsection{The relationship between tempered and Hadamard-type fractional calculus}

\begin{thm}
\label{Thm:HTtempered}
Hadamard-type fractional calculus with respect to $x$ is precisely tempered fractional calculus with respect to $\log(x)$.
\end{thm}

\begin{proof}
We recall from \cite[\S6]{fernandez-ozarslan-baleanu} that it is possible to define fractional integrals with respect to functions, not only in the Riemann--Liouville framework, but in any type of fractional calculus defined by an integral with some analytic kernel function. In the case of tempered fractional calculus, this works as follows:
\[
\prescript{T}{a}I^{\mu,s}_{g(x)}f(x)=\frac{1}{\Gamma(\mu)}\int_a^x(g(x)-g(t))^{\mu-1}e^{-s(g(x)-g(t))}f(t)g'(t)\,\mathrm{d}t.
\]
Taking $g$ to be the logarithm function, we find the following expression for the tempered fractional integral of $f(x)$ with respect to $\log(x)$:
\[
\prescript{T}{a}I^{\mu,s}_{g(x)}f(x)=\frac{1}{\Gamma(\mu)}\int_a^x\left(\log\frac{x}{t}\right)^{\mu-1}e^{-s\log\frac{x}{t}}f(t)\frac{1}{t}\,\mathrm{d}t=\frac{1}{\Gamma(\mu)}\int_a^x\left(\log\frac{x}{t}\right)^{\mu-1}\left(\frac{t}{x}\right)^s\frac{f(t)}{t}\,\mathrm{d}t,
\]
which is exactly the formula for the Hadamard-type fractional integral.

The Hadamard-type fractional derivative is obtained from composition of the fractional integral with the $n$th-order repetition of the operator $\prescript{H}{a}D^{1,s}_x$, which is defined by
\begin{align*}
\prescript{H}{a}D^{1,s}_xf(x)&=x^{-s}\left(x\frac{\mathrm{d}}{\mathrm{d}x}\right)\big(x^sf(x)\big)=x^{1-s}\Big[x^sf'(x)+sx^{s-1}f(x)\Big]=xf'(x)+sf(x) \\
&=\left(x\frac{\mathrm{d}}{\mathrm{d}x}+s\right)f(x)=\left(\frac{\mathrm{d}}{\mathrm{d}(\log x)}+s\right)f(x)=\prescript{T}{a}D^{1,s}_{\log x}f(x),
\end{align*}
namely, the tempered derivative with respect to $\log x$. Since the tempered fractional derivative is also obtained from composition of the fractional integral with the $n$th-order repetition of $\prescript{T}{a}D^{1,s}_x$, the proof is complete.
\end{proof}

\begin{thm}
\label{Thm:HTconjug}
The Hadamard-type fractional operators may be written as conjugations of the Riemann--Liouville fractional operators in two different ways:
\begin{align}
\label{HT:conjugH} \prescript{H}{a}I^{\mu,s}_x&=M_{x^s}^{-1}\circ Q_{\log x}\circ\big(\prescript{}{\log a}I_x^{\mu}\big)\circ Q_{\log x}^{-1}\circ M_{x^s} \\
\label{HT:conjugT} &=Q_{\log x}\circ M_{e^{sx}}^{-1}\circ\big(\prescript{}{\log a}I_x^{\mu}\big)\circ M_{e^{sx}}\circ Q_{\log x}^{-1},
\end{align}
and the same relations with the fractional integral $I$ replaced by a fractional derivative $D$ of either Riemann--Liouville or Caputo type, where the $Q$ and $M$ operators are defined by:
\begin{align*}
\big(Q_{g(x)}f\big)(x)=f(g(x)),\quad\quad\big(M_{g(x)}f\big)(x)=f(x)g(x).
\end{align*}
\end{thm}

\begin{proof}
It is clear from comparing Definition \ref{Def:HT} with the Hadamard fractional integral \eqref{Hadamard} that
\[
\prescript{H}{a}I^{\mu,s}_xf(x)=x^{-s}\prescript{H}{a}I^{\mu}_x\big(x^sf(x)\big),
\]
or in other words that Hadamard-type fractional calculus is the conjugation of Hadamard fractional calculus with multiplication by $x^s$:
\[
\prescript{H}{a}I^{\mu,s}_x=M_{x^s}^{-1}\circ\big(\prescript{H}{a}I^{\mu}_x\big)\circ M_{x^s}.
\]
Now Hadamard fractional calculus is Riemann--Liouville fractional calculus with respect to $\log x$, so the conjugation expression \eqref{FwrtF:conjug} gives us
\[
\prescript{H}{a}I^{\mu}_x=Q_{\log x}\circ\big(\prescript{}{\log a}I^{\mu}_x\big)\circ Q_{\log x}^{-1}.
\]
Combining the two operational identities above gives the result \eqref{HT:conjugH}.

To prove \eqref{HT:conjugT}, we use Theorem \ref{Thm:HTtempered}. Recalling from \eqref{FwrtF:conjug} that taking fractional calculus with respect to functions is equivalent to conjugation with $Q$ operators, we rewrite the result of Theorem \ref{Thm:HTtempered} as
\[
\prescript{H}{a}I^{\mu,s}_x=Q_{\log x}\circ\big(\prescript{T}{\log a}I^{\mu,s}_x\big)\circ Q_{\log x}^{-1}.
\]
And we saw in Definition \ref{Def:tempered} that tempered fractional calculus can also be written as a conjugation:
\[
\prescript{T}{\log a}I^{\mu,s}_x=M_{e^{sx}}^{-1}\circ\big(\prescript{}{a}I_x^{\mu}\big)\circ M_{e^{sx}}.
\]
Combining the two operational identities above gives the result \eqref{HT:conjugT}.
\end{proof}

\begin{rem}
The previous Theorem tells us that Hadamard-type fractional calculus is the conjugation of Riemann--Liouville fractional calculus with the operator
\[
M_{x^s}^{-1}\circ Q_{\log x}=Q_{\log x}\circ M_{e^{sx}}^{-1}.
\]
We here verify that these two operators are the same, by starting with a function $f(x)$ and examining the action of both operators:
\begin{align*}
f&\mapsto f(x). \\
M_{e^{sx}}^{-1}f&\mapsto e^{-sx}f(x); \\
Q_{\log x}\circ M_{e^{sx}}^{-1}f&\mapsto e^{-s\log x}f(\log x)=x^{-s}f(\log x). \\
Q_{\log x}f&\mapsto f(\log x); \\
M_{x^s}^{-1}\circ Q_{\log x}f&\mapsto x^{-s}f(\log x).
\end{align*}
\end{rem}

The result of Theorem \ref{Thm:HTconjug} will be very useful in understanding Hadamard-type fractional calculus, because many standard results on Hadamard-type fractional calculus now follow directly from the corresponding results from Riemann--Liouville. For example, the following results from \cite{3,Ma,Butzer} (analogous to Propositions \ref{Prop:FwrtF1} to \ref{Prop:FwrtF3} above) are now very easy to prove by using Theorem \ref{Thm:HTconjug}.

\begin{lem}\label{lem:1}
The expressions $\prescript{H}{a}I^{\mu,s}_xf(x)$ and $\prescript{HR}{a}D^{\mu,s}_xf(x)$ each define continuous (indeed, analytic) functions of $\mu$ on the respective half-lines (or half-planes in $\mathbb{C}$). If $f$ is continuous and differentiable so that both are defined, then they meet continuously at $\mu=0$, so that we can write $\prescript{H}{a}I^{\mu,s}_xf(x)=\prescript{HR}{a}D^{-\mu,s}_xf(x)$. In particular:
\begin{itemize}
	\item[(a)] $ \displaystyle\lim_{\mu \to 0^+}\prescript{H}{a}I^{\mu,s}_x f(x) = f(x) = \lim_{\mu \to 0^+}\prescript{HR}{a}D^{\mu,s}_x f(x)$;
	\item[(b)] $ \displaystyle\lim_{\mu \to (n-1)^+}\prescript{HR}{a}D^{\mu,s}_x f(x) = x^{-s}\left(x\frac{\mathrm{d}}{\mathrm{d}x}\right)^{n-1} x^s f(x), \quad\quad \lim_{\mu \to n^-}\prescript{HR}{a}D^{\mu,s}_x f(x) = x^{-s}\left(x\frac{\mathrm{d}}{\mathrm{d}x}\right)^n x^s f(x)$.
\end{itemize}
\end{lem}

\begin{lem}\label{lem:2}
For $\nu > 0$ and any $\mu,s$, the following relations hold:
	\begin{align*}
	\prescript{H}{a}I^{\mu,s}_x \left\{x^{-s} \left( \log \frac{x}{a} \right)^{\nu -1}\right\} &= \frac{\Gamma(\nu)}{\Gamma(\nu + \mu)}  x^{-s} \left( \log \frac{x}{a} \right)^{\nu + \mu -1}; \\
	\prescript{HR}{a}D^{\mu,s}_x \left\{x^{-s} \left( \log \frac{x}{a} \right)^{\nu -1}\right\} &= \frac{\Gamma(\nu)}{\Gamma(\nu - \mu)}  x^{-s} \left( \log \frac{x}{a} \right)^{\nu - \mu -1}.
	\end{align*}
\end{lem}

\begin{lem}\label{thm1}
The Hadamard-type fractional integrals and derivatives have the following semigroup properties:
	\begin{align*}
	\prescript{H}{a}I^{\nu,s}_x \prescript{H}{a}I^{\mu,s}_x f(x) &= \prescript{H}{a}I^{\nu + \mu,s}_x f(x),\quad\quad\mu,\nu>0; \\
	\prescript{HR}{a}D^{\mu,s}_x \prescript{H}{a}I^{\nu,s}_x f(x) &= \prescript{HR}{a}D^{\mu - \nu,s}_x f(x),\quad\quad\nu>0; \\
	x^{-s}\left(x\frac{\mathrm{d}}{\mathrm{d}x}\right)^n x^s \prescript{HR}{a}D^{\mu,s}_x f(x) &= \prescript{HR}{a}D^{n+\mu,s}_x f(x),
	\end{align*}
	where in the last two identities the operators denoted by $D$ may be either fractional derivatives or fractional integrals according to the sign of the index.
\end{lem}

\subsection{Tempered and Hadamard-type fractional calculus with respect to functions}

Theorem \ref{Thm:HTtempered} provides a connection between tempered fractional calculus and Hadamard-type fractional calculus. It is also possible to unite both types of fractional calculus into a single, more general, definition. If Hadamard-type fractional calculus is the same as tempered fractional calculus with respect to $\log x$, then Hadamard-type fractional calculus with respect to an arbitrary (monotonic, $C^1$) function $\Psi(x)$ is the same as tempered fractional calculus with respect to an arbitrary (monotonic, $C^1$) function $\phi(x)=\log\Psi(x)$. We formalise this more general type of fractional calculus, covering both tempered and Hadamard-type as special cases, in the following definition.

    \begin{defn}\label{defn:3}
     Let $\mu\in\mathbb{R}$ with $\mu>0$ (or $\mu\in\mathbb{C}$ with $\mathrm{Re}(\mu)>0$), and $s\in\mathbb{C}$. Let $f$ be an integrable function defined on $[a,b]$ where $a<b$ in $\mathbb{R}$, and $\Psi\in C^1([a,b])$ be a positive increasing function such that $ \Psi'(x) \neq 0 $ for all $ x \in [a,b] $. Then, the Hadamard-type fractional integral of $f$ with respect to $\Psi$, or the tempered fractional integral of $f$ with respect to $\log\circ\Psi$, with order $\mu$ and parameter $s$, is defined as
	\begin{equation}
	\prescript{H}{a}I^{\mu,s}_{\Psi(x)} f(x) = \prescript{T}{a}I^{\mu,s}_{\log\Psi(x)} f(x) = \frac{1}{\Gamma(\mu)} \int_{a}^{x} \left(\frac{\Psi(t)}{\Psi(x)}\right)^{s} \left(\log \frac{\Psi(x)}{\Psi(t)}\right)^ {\mu - 1} f(t) \frac{\Psi '(t)}{\Psi(t)}\,\mathrm{d}t,\quad\quad x\in[a,b].
	\end{equation}

	Writing $n=\lfloor\mu\rfloor+1$ (or $n=\lfloor\mathrm{Re}(\mu)\rfloor+1$) so that $ n-1 \leq \mu < n\in\mathbb{Z}^+$, the Hadamard-type fractional derivative of $f$ with respect to $\Psi$, or the tempered fractional derivative of $f$ with respect to $\log\circ\Psi$, with order $\mu$ and parameter $s$, is defined as (in the Riemann--Liouville sense)
	\begin{equation}
	\prescript{HR}{a}D^{\mu,s}_{\Psi(x)} f(x) = \prescript{TR}{a}D^{\mu,s}_{\log\Psi(x)} f(x) = \prescript{H}{a}D^{n,s}_{\Psi(x)} \prescript{H}{a}I^{n - \mu,s}_{\Psi(x)} f(x), \quad\quad x\in[a,b],
	\end{equation}
or (in the Caputo sense)
	\begin{equation}
	\prescript{HC}{a}D^{\mu,s}_{\Psi(x)} f(x) = \prescript{TC}{a}D^{\mu,s}_{\log\Psi(x)} f(x) = \prescript{H}{a}I^{n - \mu,s}_{\Psi(x)} \prescript{H}{a}D^{n,s}_{\Psi(x)} f(x), \quad\quad x\in[a,b],
	\end{equation}
where the $n$th-order derivative is defined by
	\begin{equation}
	\prescript{H}{a}D^{n,s}_{\Psi(x)}f(x) = \Psi(x)^{-s} \left(\frac{\Psi(x)}{\Psi'(x)}\cdot\frac{\mathrm{d}}{\mathrm{d}x}\right)^{n} \Big[\Psi(x)^{s}f(x)\Big].
	\end{equation}
\end{defn}

\begin{rem}
	It can be seen that the fractional operators defined in the Definition \ref{defn:3} generalise several existing fractional operators.
\begin{itemize}
\item For $ \Psi(x)=x $, we obtain the Hadamard-type fractional calculus.
\item For $\Psi(x)=e^x$, we obtain the tempered fractional calculus.
\item More generally, writing $ \Psi(x) = e^{\phi(x)} $, these operators are precisely the tempered fractional operators with respect to the monotonic $C^1$ function $\phi$.
\end{itemize}
For $ s=0 $, we obtain the Hadamard fractional operators with respect to $\Psi$. This also gives rise to interesting special cases.
\begin{itemize}
\item For $s=0$ and $ \Psi(x)=e^x $, we obtain the original Riemann--Liouville and Caputo fractional calculus.
\item For $s=0$ and $ \Psi(x)=e^{x^{\rho}/\rho}$, we obtain the so-called Katugampola fractional calculus (fractional operators with respect to $x^{\rho}$, as defined by Erdelyi in 1964).
\item For $s=0$ and $ \Psi(x) = x $, we obtain the original Hadamard fractional calculus.
\item More generally, for $s=0$ and $ \Psi(x) = e^{\phi(x)} $, we obtain the Riemann--Liouville and Caputo fractional operators of a function with respect to another function $\phi$.
\end{itemize}
Finally, we note that the operators we have defined here are a generalisation of the so-called ``generalised substantial fractional operators'' introduced recently in \cite{fahad-urrehman}. This is because substantial fractional calculus is identical to tempered fractional calculus, and the operators defined in \cite{fahad-urrehman} are precisely tempered fractional operators with respect to power functions -- a special case of tempered fractional operators with respect to general increasing functions, which we have defined here.
\end{rem}

\begin{prop}
\label{Prop:ourconjug}
The Hadamard-type or tempered fractional operators with respect to a function may be written as conjugations of the Riemann--Liouville fractional operators as follows:
\[
\prescript{H}{a}I^{\mu,s}_{\Psi(x)}=Q_{\log\Psi(x)}\circ M_{e^{sx}}^{-1}\circ\big(\prescript{}{\log\Psi(a)}I_x^{\mu}\big)\circ M_{e^{sx}}\circ Q_{\log\Psi(x)}^{-1},
\]
and the same relations with the fractional integral $I$ replaced by a fractional derivative $D$ of either Riemann--Liouville or Caputo type, where the $Q$ and $M$ operators are defined as in Theorem \ref{Thm:HTconjug}.
\end{prop}

\begin{proof}
This follows from combining \eqref{HT:conjugT} from Theorem \ref{Thm:HTconjug} with the known identity \eqref{FwrtF:conjug}:
\begin{align*}
\prescript{H}{a}I^{\mu,s}_{\Psi(x)}&=Q_{\Psi(x)}\circ\Big(\prescript{H}{\Psi(a)}I_x^{\mu,s}\Big)\circ Q_{\Psi(x)}^{-1} \\
&=Q_{\Psi(x)}\circ\Big(Q_{\log x}\circ M_{e^{sx}}^{-1}\circ\Big(\prescript{}{\log\Psi(a)}I_x^{\mu}\Big)\circ M_{e^{sx}}\circ Q_{\log x}^{-1}\Big)\circ Q_{\Psi(x)}^{-1},
\end{align*}
and clearly $Q_{h(x)}\circ Q_{g(x)}=Q_{g(h(x))}$. Alternatively, we can use the fact that Hadamard-type fractional calculus with respect to $\Psi$ is tempered fractional calculus with respect to $\log\circ\Psi$:
\begin{align*}
\prescript{H}{a}I^{\mu,s}_{\Psi(x)}&=\prescript{T}{a}I^{\mu,s}_{\log\Psi(x)}=Q_{\log\Psi(x)}\circ\Big(\prescript{T}{\log\Psi(a)}I_x^{\mu,s}\Big)\circ Q_{\log\Psi(x)}^{-1} \\
&=Q_{\log\Psi(x)}\circ\Big(M_{e^{sx}}^{-1}\circ\Big(\prescript{}{\log\Psi(a)}I_x^{\mu}\Big)\circ M_{e^{sx}}\Big)\circ Q_{\log\Psi(x)}^{-1}.
\end{align*}
\end{proof}

\begin{cor}
\label{Cor:ourconjug}
The Hadamard-type or tempered fractional operators with respect to a function may be written as
\[
\prescript{H}{a}I^{\mu,s}_{\Psi(x)}f(x)=\Psi(x)^{-s}\prescript{}{a}I_{\log\Psi(x)}^{\mu}\big[\Psi(x)^sf(x)\big],
\]
and the same relations with the fractional integral $I$ replaced by a fractional derivative $D$ of either Riemann--Liouville or Caputo type.
\end{cor}

\begin{proof}
It will be enough to show that
\[
Q_{\log\Psi(x)}\circ M_{e^{sx}}^{-1}=M_{\Psi(x)^{s}}^{-1}Q_{\log\Psi(x)},
\]
since the result will then follow from the identity \eqref{FwrtF:conjug} together with the result of Proposition \ref{Prop:ourconjug}. We check how these operators work on a function $f(x)$:
\begin{align*}
f&\mapsto f(x). \\
M_{e^{sx}}^{-1}f&\mapsto e^{-sx}f(x); \\
Q_{\log\Psi(x)}\circ M_{e^{sx}}^{-1}f&\mapsto e^{-s\log\Psi(x)}f(\log\Psi(x))=\Psi(x)^{-s}f(\log\Psi(x)). \\
Q_{\log\Psi(x)}f&\mapsto f(\log\Psi(x)); \\
M_{\Psi(x)^s}^{-1}\circ Q_{\log\Psi(x)}f&\mapsto \Psi(x)^{-s}f(\log\Psi(x)).
\end{align*}
\end{proof}

Once again, given the conjugation relations with the original Riemann--Liouville fractional calculus, it is easy to prove some basic properties (analogous to Propositions \ref{Prop:FwrtF1} to \ref{Prop:FwrtF3} above) of Hadamard-type fractional calculus with respect to a function.

\begin{prop}
\label{Prop:ourlem1}
The expressions $\prescript{H}{a}I^{\mu,s}_{\Psi(x)}f(x)$ and $\prescript{HR}{a}D^{\mu,s}_{\Psi(x)}f(x)$ each define continuous (indeed, analytic) functions of $\mu$ on the respective half-lines (or half-planes in $\mathbb{C}$). If $f$ is continuous and differentiable so that both are defined, then they meet continuously at $\mu=0$, so that we can write $\prescript{H}{a}I^{\mu,s}_{\Psi(x)}f(x)=\prescript{HR}{a}D^{-\mu,s}_{\Psi(x)}f(x)$. In particular:
\begin{itemize}
	\item[(a)] $ \displaystyle\lim_{\mu \to 0^+}\prescript{H}{a}I^{\mu,s}_{\Psi(x)} f(x) = f(x) = \lim_{\mu \to 0^+}\prescript{HR}{a}D^{\mu,s}_{\Psi(x)} f(x)$;
	\item[(b)] $ \displaystyle\lim_{\mu \to (n-1)^+}\prescript{HR}{a}D^{\mu,s}_x f(x) = \Psi(x)^{-s}\left(\frac{\Psi(x)}{\Psi'(x)}\cdot\frac{\mathrm{d}}{\mathrm{d}x}\right)^{n-1} \Psi(x)^s f(x)$, \\ $\displaystyle {\color{white}m}\lim_{\mu \to n^-}\prescript{HR}{a}D^{\mu,s}_{\Psi(x)} f(x) = \Psi(x)^{-s}\left(\frac{\Psi(x)}{\Psi'(x)}\cdot\frac{\mathrm{d}}{\mathrm{d}x}\right)^n \Psi(x)^s f(x)$.
\end{itemize}
\end{prop}

\begin{proof}
The main result is clear from the conjugation relations proved above. We have
\[
\prescript{H}{a}I^{\mu,s}_{\Psi(x)}f(x)=\Psi(x)^{-s}\prescript{}{a}I_{\log\Psi(x)}^{\mu}\big[\Psi(x)^sf(x)\big],
\]
and the result of applying $\prescript{}{a}I_{\log\Psi(x)}^{\mu}$ is a continuous (analytic) function of $\mu$ (by Proposition \ref{Prop:FwrtF1} above), so we can say the same about the result of applying $\prescript{H}{a}I^{\mu,s}_{\Psi(x)}$. Then (a) is clear from continuity, and (b) is a special case of the above identity since $\frac{\mathrm{d}}{\mathrm{d}(\log\Psi(x))}=\frac{\Psi(x)}{\Psi'(x)}\cdot\frac{\mathrm{d}}{\mathrm{d}x}$.
\end{proof}

\begin{prop}
\label{Prop:ourlem2}
The Hadamard-type or tempered fractional operators with respect to a function have the following semigroup properties:
	\begin{align*}
	\prescript{H}{a}I^{\nu,s}_{\Psi(x)} \prescript{H}{a}I^{\mu,s}_{\Psi(x)} f(x) &= \prescript{H}{a}I^{\nu + \mu,s}_{\Psi(x)} f(x),\quad\quad\mu,\nu>0; \\
	\prescript{HR}{a}D^{\mu,s}_{\Psi(x)} \prescript{H}{a}I^{\nu,s}_{\Psi(x)} f(x) &= \prescript{HR}{a}D^{\mu - \nu,s}_{\Psi(x)} f(x),\quad\quad\nu>0; \\
	\Psi(x)^{-s}\left(\frac{\Psi(x)}{\Psi'(x)}\cdot\frac{\mathrm{d}}{\mathrm{d}x}\right)^n \Psi(x)^s \prescript{HR}{a}D^{\mu,s}_{\Psi(x)} f(x) &= \prescript{HR}{a}D^{n+\mu,s}_{\Psi(x)} f(x),
	\end{align*}
	where in the last two identities the operators denoted by $D$ may be either fractional derivatives or fractional integrals according to the sign of the index.
\end{prop}

\begin{proof}
This follows directly from the conjugation relation given by Proposition \ref{Prop:ourconjug}, combined with the semigroup properties of the original Riemann--Liouville integrals and derivatives.
\end{proof}

\begin{prop}
\label{Prop:ourlem3}
For $\nu>0$ and any $\mu,s$, the following relations hold:
	\begin{align*}
	\prescript{H}{a}I^{\mu,s}_{\Psi(x)} \left\{\Psi(x)^{-s} \left( \log \frac{\Psi(x)}{\Psi(a)} \right)^{\nu -1}\right\} &= \frac{\Gamma(\nu)}{\Gamma(\nu + \mu)}  \Psi(x)^{-s} \left( \log \frac{\Psi(x)}{\Psi(a)} \right)^{\nu + \mu -1}; \\
	\prescript{HR}{a}D^{\mu,s}_{\Psi(x)} \left\{\Psi(x)^{-s} \left( \log \frac{\Psi(x)}{\Psi(a)} \right)^{\nu -1}\right\} &= \frac{\Gamma(\nu)}{\Gamma(\nu - \mu)}  \Psi(x)^{-s} \left( \log \frac{\Psi(x)}{\Psi(a)} \right)^{\nu - \mu -1}.
	\end{align*}
\end{prop}

\begin{proof}
This follows directly from Corollary \ref{Cor:ourconjug} together with Proposition \ref{Prop:FwrtF3} for Riemann--Liouville fractional calculus with respect to functions.
\end{proof}

\begin{rem}
Combining the second part of Proposition \ref{Prop:ourlem2} with the first part of Proposition \ref{Prop:ourlem1}, it is clear that, on the space of continuous
functions, the Hadamard-type fractional differential operator with respect to a function in the Riemann--Liouville sense is the left inverse of the corresponding integral operator. This is no longer necessarily true for the Hadamard-type fractional differential operator with respect to a function in the Caputo sense. Just like in the original fractional calculus, the Riemann--Liouville type derivative is a left inverse of the integral but the Caputo type derivative is not. Indeed, it was shown in \cite[Counter-Example 1]{cichon-salem2} that there is a H\"olderian but not absolutely continuous function $f$ such that $\prescript{HC}{0}D^{\mu,0}_{e^x}\prescript{H}{0}D^{\mu,0}_{e^x}f(x)\neq f(x)$ for any $\mu\in(0,1)$. For further discussion of the appropriate function spaces for the generalised operators defined here, we refer the reader to Section \ref{Sec:fnspace} below.
\end{rem}

Another theorem, a continuation of Proposition \ref{Prop:ourlem2} above, concerns taking integrals of derivatives in Hadamard-type fractional calculus with respect to functions. This is analogous to the result for Riemann--Liouville fractional calculus \cite{Samko},
\begin{equation}
\label{NewtonLeibniz:RL}
\prescript{}{a}I^{\mu}_x\prescript{R}{a}D^{\mu}_xf(x)=f(x)-\sum_{k=1}^n\frac{(x-a)^{\mu-k}}{\Gamma(\mu-k+1)}\lim_{x\rightarrow a^+}\prescript{R}{a}D^{\mu-k}_xf(x),
\end{equation}
and it is less easy to prove from operator theory than Proposition \ref{Prop:ourlem2}, because here there is no longer an actual semigroup property.

\begin{thm}\label{theorem2B}
Let $\mu > 0$, $n = \lfloor\mu\rfloor + 1$, $s \in \R$, $\Psi$ an increasing positive function on $[a,b]\subset\mathbb{R}$. If $ f(x) $ is a function such that $\prescript{H}{a}I^{n-\mu,s}_{\Psi(x)}f\in AC_{\delta^{\Psi}, s}^{n} [a,b]$, then we have
\[
	\prescript{H}{a}I^{\mu,s}_{\Psi(x)} \prescript{HR}{a}D^{\mu,s}_{\Psi(x)} f(x) = f(x) - \left(\frac{\Psi(a)}{\Psi(x)}\right)^s\sum_{k=1}^{n} \frac{1}{\Gamma(\mu-k+1)} \left( \log \frac{\Psi(x)}{\Psi(a)} \right)^{\mu-k}\lim_{x\to a^+} \prescript{HR}{a}D^{\mu-k,s}_{\Psi(x)} f(x).
	\]
In particular, for $ 0<\mu<1 $ we have
	\[
	\prescript{H}{a}I^{\mu,s}_{\Psi(x)} \prescript{HR}{a}D^{\mu,s}_{\Psi(x)} f(x) = f(x) - \frac{1}{\Gamma(\mu)} \left(\frac{\Psi(a)}{\Psi(x)}\right)^s \left( \log \frac{\Psi(x)}{\Psi(a)} \right)^{\mu-1}\lim_{x\to a^+} \prescript{H}{a}I^{1-\mu,s}_{\Psi(x)} f(x).
	\]
\end{thm}

\begin{proof}
\allowdisplaybreaks
It is convenient to rewrite the Hadamard-type fractional operators with respect to functions using the form of Proposition \ref{Prop:ourconjug}.  We then make use of \eqref{NewtonLeibniz:RL}:
\begin{align*}
\prescript{H}{a}I^{\mu,s}_{\Psi(x)} \prescript{HR}{a}D^{\mu,s}_{\Psi(x)} f(x) &= Q_{\log\Psi(x)}\circ M_{e^{sx}}^{-1}\circ\Big(\prescript{}{\log\Psi(a)}I_x^{\mu}\circ\prescript{R}{\log\Psi(a)}D_x^{\mu}\Big)\circ M_{e^{sx}}\circ Q_{\log\Psi(x)}^{-1}f(x) \\
&\hspace{-2cm}= Q_{\log\Psi(x)}\circ M_{e^{sx}}^{-1}\circ\Big(\prescript{}{\log\Psi(a)}I_x^{\mu}\circ\prescript{R}{\log\Psi(a)}D_x^{\mu}\Big)\Big[e^{sx}f\big(\Psi^{-1}(e^x)\big)\Big] \\
&\hspace{-2cm}= Q_{\log\Psi(x)}\circ M_{e^{sx}}^{-1}\left[e^{sx}f\big(\Psi^{-1}(e^x)\big)-\sum_{k=1}^n\frac{(x-\log\Psi(a))^{\mu-k}}{\Gamma(\mu-k+1)}\lim_{x\rightarrow \log\Psi(a)^+}\prescript{R}{a}D^{\mu-k}_x\Big[e^{sx}f\big(\Psi^{-1}(e^x)\big)\Big]\right] \\
&\hspace{-2cm}=Q_{\log\Psi(x)}\circ M_{e^{sx}}^{-1}\circ M_{e^{sx}}\circ Q_{\log\Psi(x)}^{-1}f(x) \\
&-\sum_{k=1}^n\frac{Q_{\log\Psi(x)}\circ M_{e^{sx}}^{-1}\Big[(x-\log\Psi(a))^{\mu-k}\Big]}{\Gamma(\mu-k+1)}\lim_{x\rightarrow \log\Psi(a)^+}\Big[\prescript{R}{a}D^{\mu-k}_x\circ M_{e^{sx}}\circ Q_{\log\Psi(x)}^{-1}f(x)\Big] \\
&\hspace{-2cm}=f(x)-\sum_{k=1}^n\frac{e^{-s\log\Psi(x)}\big(\log\Psi(x)-\log\Psi(a)\big)^{\mu-k}}{\Gamma(\mu-k+1)} \\
&\hspace{2cm}\times\lim_{x\rightarrow a^+} \Big(\Psi(x)^s \Big[Q_{\log\Psi(x)}\circ M_{e^{sx}}^{-1}\circ\prescript{R}{a}D^{\mu-k}_x\circ M_{e^{sx}}\circ Q_{\log\Psi(x)}^{-1}f(x)\Big]\Big) \\
&\hspace{-2cm}=f(x)-\sum_{k=1}^n\frac{\Psi(x)^{-s}}{\Gamma(\mu-k+1)} \left(\log\frac{\Psi(x)}{\Psi(a)}\right)^{\mu-k} \Psi(a)^s \lim_{x\rightarrow a^+} \Big[\prescript{HR}{a}D^{\mu-k}_{\Psi(x)}f(x)\Big].
\end{align*}
	Hence, we get our desired result.
\end{proof}

\begin{figure}[H]
	\centering
	\subfloat[  $ \Psi(x)= \sqrt{x} $, $0.1\leq\mu\leq 0.9$.]
	{\label{Fig1a}\includegraphics[width=0.33\textwidth]{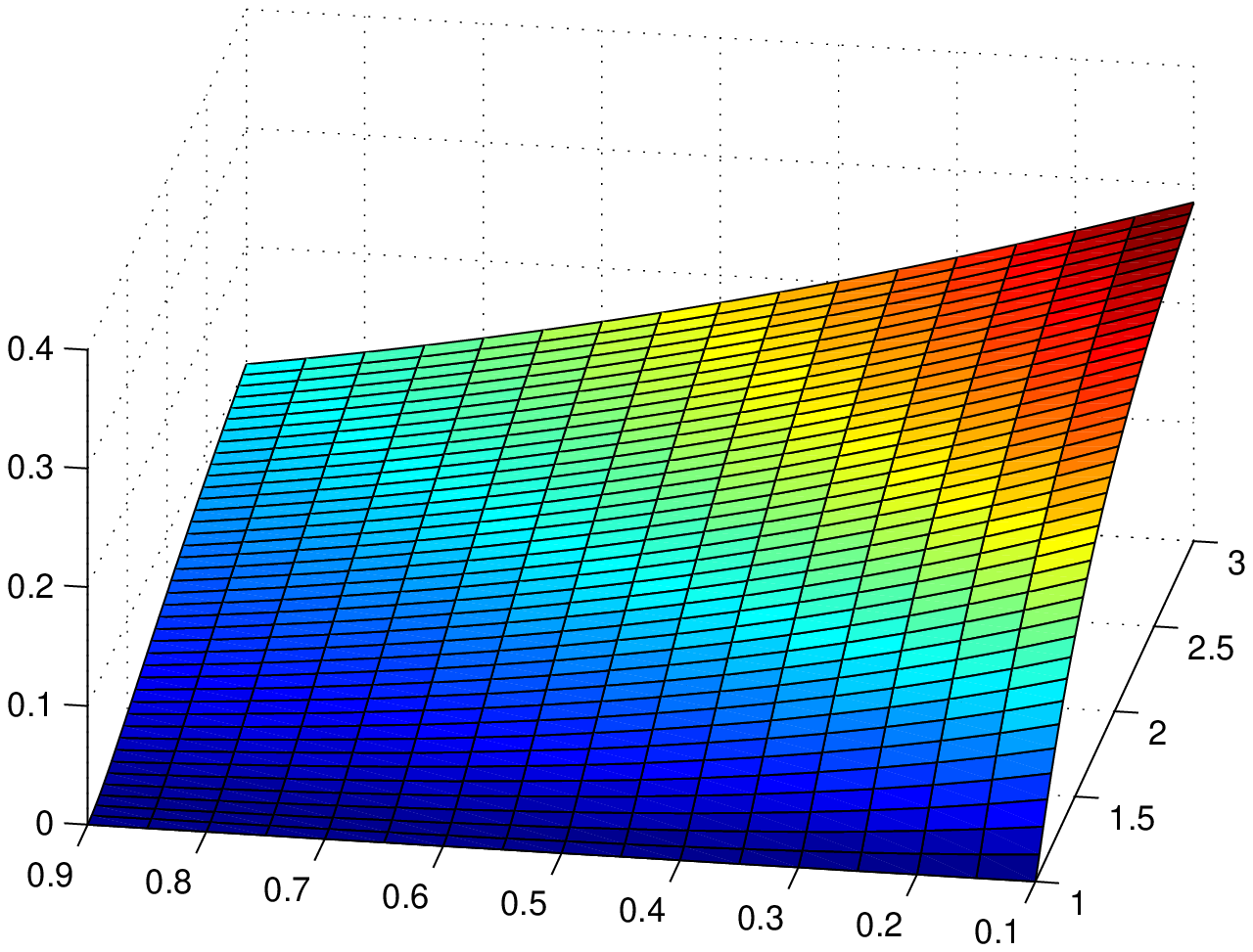}}
	\subfloat[ $ \Psi(x)= {x} $, $0.1\leq\mu\leq 0.9$.]
	{\label{Fig1b}\includegraphics[width=0.33\textwidth]{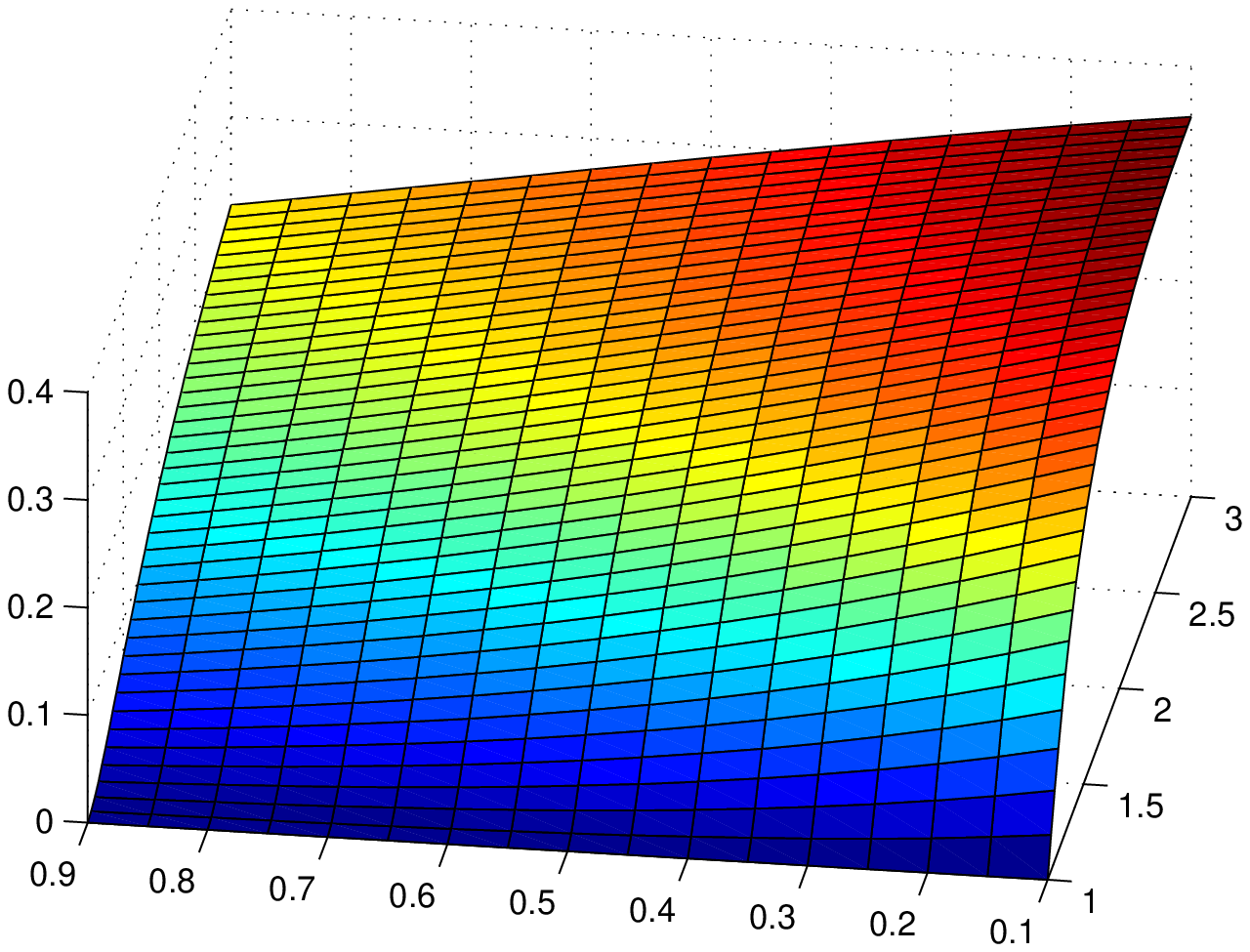}}
	\subfloat[ $ \Psi(x)= {x}^{2} $, $0.1\leq\mu\leq 0.9$.]
	{\label{Fig1c}\includegraphics[width=0.33\textwidth]{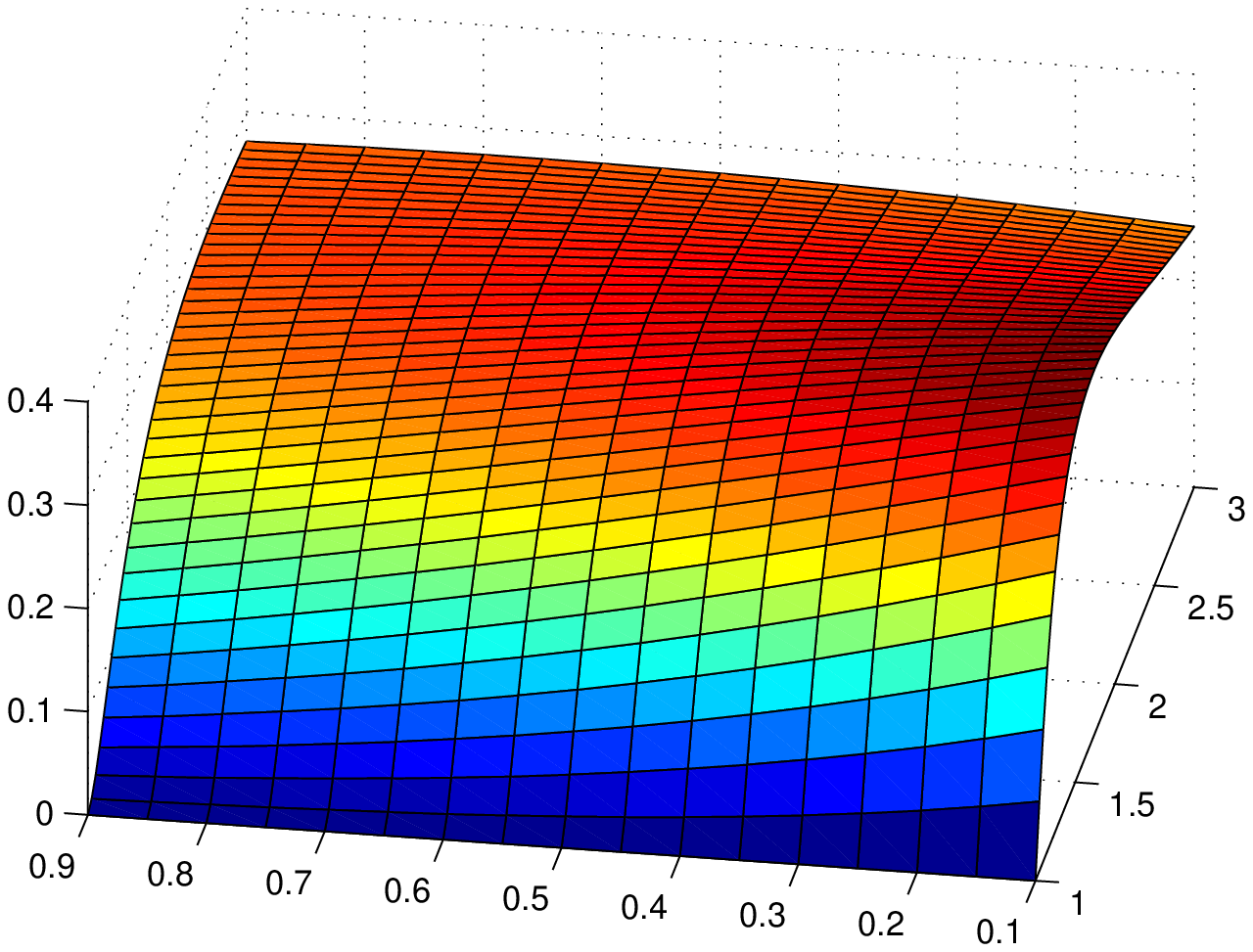}}
	\caption{Hadamard-type fractional integrals $\prescript{H}{1}I^{\mu,1}_{\Psi(x)}f(x)$ with respect to $\Psi(x)$ of the function $f(x)=\Psi(x)^{-1} \log\frac{\Psi(x)}{\Psi(1)}$, with different functions $\Psi$. }
	\label{Fig1}
\end{figure}
\begin{figure}[H]
	\centering
	\subfloat[ $ \Psi(x)= \sqrt{x} $, $0.1\leq\mu\leq 0.9$.]
	{\label{Fig2a}\includegraphics[width=0.33\textwidth]{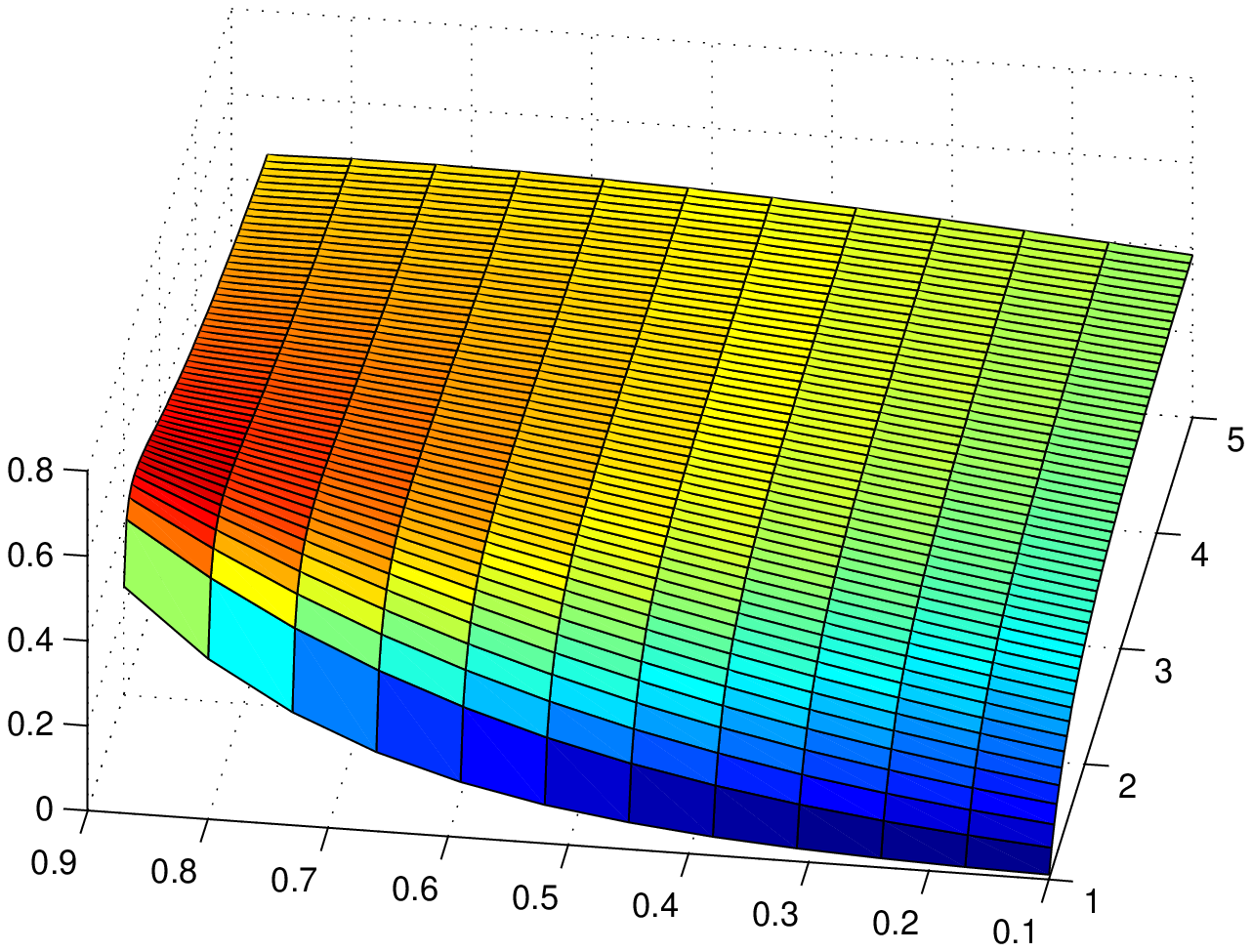}}
	\subfloat[ $ \Psi(x)= {x} $, $0.1\leq\mu\leq 0.9$.]
	{\label{Fig2b}\includegraphics[width=0.33\textwidth]{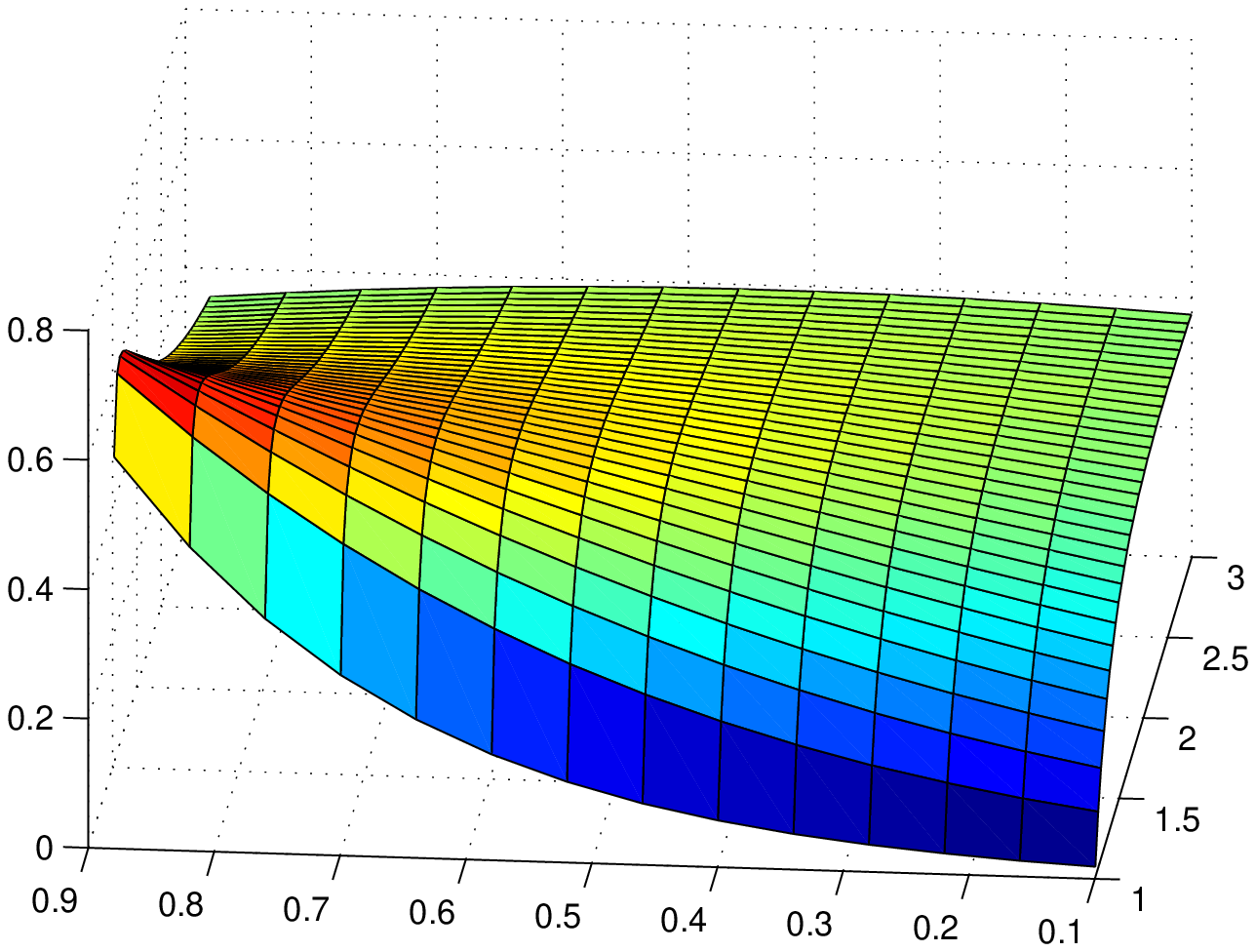}}
	\subfloat[ $ \Psi(x)= {x}^{2} $, $0.1\leq\mu\leq 0.9$.]
	{\label{Fig2c}\includegraphics[width=0.33\textwidth]{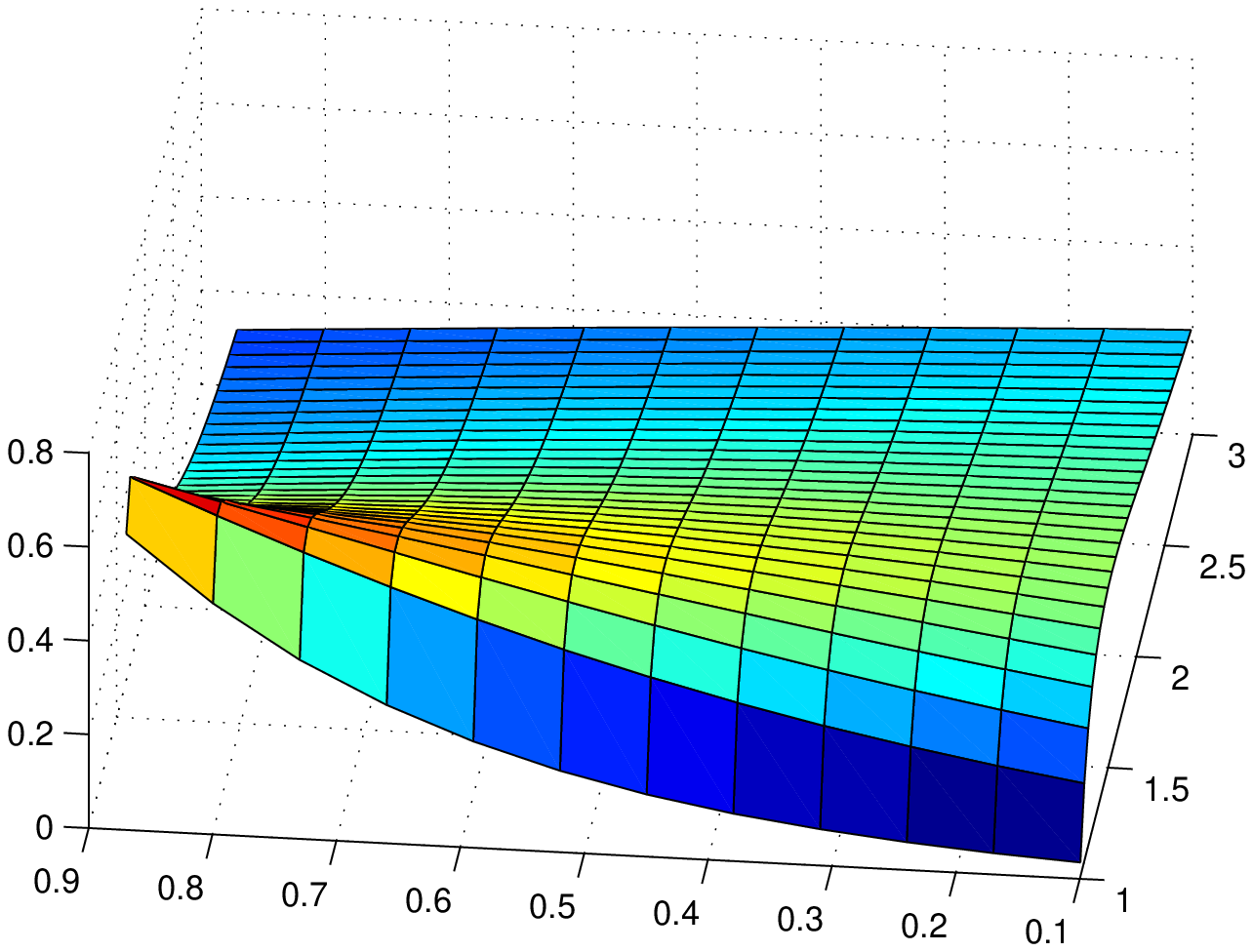}}
	\caption{Hadamard-type fractional derivatives $\prescript{HR}{1}D^{\mu,1}_{\Psi(x)}f(x)$ with respect to $\Psi(x)$ of the function $f(x)=\Psi(x)^{-1} \log\frac{\Psi(x)}{\Psi(1)}$, with different functions $\Psi$. }
	\label{Fig2}
\end{figure}

In Figure \ref{Fig1} and Figure \ref{Fig2}, we plot some example fractional integrals and derivatives, respectively, of the function from Proposition \ref{Prop:ourlem3}, in the case $\nu=2$, $a=1$, $s=1$, for a few example functions $\Psi(x)$. These graphs plot the expressions $\prescript{H}{a}I^{\mu,s}_{\Psi(x)}f(x)$ and $\prescript{HR}{a}D^{\mu,s}_{\Psi(x)}f(x)$ respectively against the independent variables $x$ and $\mu$, where in each case $f(x)=\Psi(x)^{-1} \log\frac{\Psi(x)}{\Psi(1)}$ and $\Psi$ is a power function as specified in the subfigure caption.

\section{Function spaces for Hadamard-type fractional calculus with respect to $\Psi(x)$} \label{Sec:fnspace}

\subsection{Function spaces for the fractional integral operator}

In this section, we discuss conditions under which the Hadamard-type fractional integral operator $\prescript{H}{a}I^{\mu,s}_{\Psi(x)}$ with respect to a function is bounded in a particular function space $X_{\Psi, c}^{p} (a, b)$ which we define, following \cite{Butzer} for Hadamard-type fractional calculus, as follows.

\begin{defn}
Let $c \in \R$ and $1 \leq p \leq \infty$. The space $X_{\Psi, c}^{p} (a, b)$ is defined to consist of all Lebesgue measurable functions $f:[a,b]\rightarrow\mathbb{C}$ for which $ \|f\|_{X_{\Psi, c}^{p}} < \infty$, where the norm is defined by
\begin{equation} \label{generalizednorm}
||f||_{X_{\Psi, c}^{p}} = \left( \int_{a}^{b} \left|\left\{\Psi(x)\right\}^{c} f(x)\right|^{p} \frac{\Psi^{\prime}(x)}{\Psi(x)} \,\mathrm{d}x \right)^{\frac{1}{p}}\ \ \ \text{for }c \in \R, 1 \leq p < \infty,
\end{equation}
and
\begin{equation} \label{normforpinfinity}
||f||_{X_{\Psi, c}^{\infty}} =\esssup_{a \leq x \leq b} \left(\left\{\Psi(x)\right\}^{c} \left|f(x)\right| \right), \ \ \ \text{for }c \in \R. 
\end{equation}
\end{defn}

If we consider $ c=\frac{1}{p} $ and $ \Psi(x)=x $, then the space $X_{\Psi, c}^{p} (a, b)$ coincides with the space $ L^{p}(a,b) $ with
\[
\|f\|_{p} = \left( \int_{a}^{b} \left|f(x)\right|^{p}  dx \right)^{\frac{1}{p}} \ \ \ \text{for }c \in \R,1 \leq p < \infty,
\]
and
\[
\|f\|_{\infty} =\esssup_{a \leq x \leq b} \left| f(x)\right|  \ \ \ \text{for }c \in \R.
\]

In the theorem below, we prove that, for a positive increasing function $ \Psi $ and for any $ s \geq c $, the Hadamard-type fractional integral operator of any positive order $\mu$ and parameter $s$ with respect to $\Psi$ is well-defined on the space $X_{\Psi, c}^{p} (a, b)$.

\begin{thm}\label{thm4.1}
	Let $\mu > 0$, $1 \leq p \leq \infty$, $a < b$ and $s\geq c$ in $\mathbb{R}$, and $\Psi$ be a positive increasing function. Then the operator $\prescript{H}{a}I^{\mu,s}_{\Psi(x)}$ is bounded in the space $X_{\Psi, c}^{p}(a, b)$ and
	\begin{equation}
	\left\| \prescript{H}{a}I^{\mu,s}_{\Psi(x)} f \right\|_{X_{\Psi, c}^{p}} \leq K \|f\|_{X_{\Psi, c}^{p}},
	\end{equation}
	where the constant $K$ is defined by
	\begin{equation} \label{Knorm}
	K =\begin{cases}
	\displaystyle\frac{1}{\Gamma(\mu +1)} \left(\log \frac{\Psi(b)}{\Psi(a)} \right)^{\mu} \ \ \ &\text{for } s = c, \\
	\displaystyle\frac{1}{\Gamma(\mu)} (s - c)^{-\mu} \gamma \left(\mu,  (s - c) \log  \frac{\Psi(b)}{\Psi(a)} \right) \ \ \ &\text{for } s> c.
	\end{cases}
	\end{equation}
\end{thm}

\begin{proof}
	First we discuss the case $1 \leq p < \infty$. Using the definition of $\prescript{H}{a}I^{\mu,s}_{\Psi(x)}$ and Eq. \eqref{generalizednorm}, we find
	\[
	\left\| \prescript{H}{a}I^{\mu,s}_{\Psi(x)} f\right\|_{X_{\Psi, c}^{p}} = \left( \int_{a}^{b} \Psi(x)^{cp} \left| \frac{1}{\Gamma(\mu)} \int_{a}^{x} \left(\frac{\Psi(t)}{\Psi(x)} \right)^{s} \left(\log  \frac{\Psi(x)}{\Psi(t)} \right)^{\mu - 1} f(t) \frac{\Psi'(t)}{\Psi(t)}\,\mathrm{d}t \right|^{p} \frac{\Psi'(x)}{\Psi(x)}\,\mathrm{d}x  \right)^{\frac{1}{p}}.
	\]
	Making the substitution $\Psi(t) = \frac{\Psi(x)}{\Psi(u)}$ in the inner integral, which gives $\frac{\Psi'(t)}{\Psi(t)}\,\mathrm{d}t=-\frac{\Psi'(u)}{\Psi(u)}\,\mathrm{d}u$, we get	
	\begin{align*}
	&\left\| \prescript{H}{a}I^{\mu,s}_{\Psi(x)} f\right\|_{X_{\Psi, c}^{p}} \\
	&\hspace{1cm}= \left( \int_{a}^{b} \left| \frac{\Psi(x)^c}{\Gamma(\mu)} \int_{\Psi^{-1} \left(\frac{\Psi(x)}{\Psi(a)}\right)}^{\Psi^{-1}(1)} \Psi(u)^{-s} \left(\log\Psi(u)\right)^{\mu - 1} f\left( \Psi^{-1} \left(\frac{\Psi(x)}{\Psi(u)}\right) \right) \left(-\frac{\Psi'(u)}{\Psi(u)}\right)\,\mathrm{d}u \right|^{p} \frac{\Psi'(x)}{\Psi(x)}\,\mathrm{d}x  \right)^{\frac{1}{p}} \\
	&\hspace{1cm}=  \left( \int_{a}^{b} \left| \int_{\Psi^{-1}(1)}^{\Psi^{-1} \left(\frac{\Psi(x)}{\Psi(a)}\right)} \frac{\Psi(u)^{-s}}{\Gamma(\mu)}  \left(\log \Psi(u)\right)^{\mu - 1} \frac{\Psi(x)^{c-\frac{1}{p}}}{\Psi'(x)^{-\frac{1}{p}}} f\left( \Psi^{-1} \left(\frac{\Psi(x)}{\Psi(u)}\right) \right) \frac{\Psi'(u)}{\Psi(u)}\,\mathrm{d}u  \right|^{p} \,\mathrm{d}x \right)^{\frac{1}{p}}.
	\end{align*}
	Since $ f(x) \in X_{\Psi, c}^{p}(a,b)$, it follows that $ \frac{\Psi(x)^{c-\frac{1}{p}}}{\Psi'(x)^{-\frac{1}{p}}} f(x) \in  L^{p}(a,b)$, and hence by application of the generalized Minkowski inequality, we have
	\begin{align*}
	\left\| \prescript{H}{a}I^{\mu,s}_{\Psi(x)} f\right\|_{X_{\Psi, c}^{p}} &\leq \frac{1}{\Gamma(\mu)} \int_{\Psi^{-1}(1)}^{\Psi^{-1} \left(\frac{\Psi(b)}{\Psi(a)}\right)}  \left( \int_{\Psi ^{-1} \left(\Psi(u) \Psi(a)\right)}^{b} \Psi(x)^{cp} \left| f\left( \Psi^{-1} \left(\frac{\Psi(x)}{\Psi(u)}\right) \right) \right|^{p} \frac{\Psi'(x)}{\Psi(x)}\,\mathrm{d}x \right)^{\frac{1}{p}}  \\
	&\hspace{3cm}\times \Psi(u)^{-s - 1} \left(\log \Psi(u)\right)^{\mu - 1} \Psi'(u) \,\mathrm{d}u \\
	&= \frac{1}{\Gamma(\mu)} \int_{\Psi^{-1}(1)}^{\Psi^{-1} \left(\frac{\Psi(b)}{\Psi(a)}\right)} \left( \int_{a}^{\Psi^{-1}\left(\frac{\Psi(b)}{\Psi(u)} \right)} \big| \Psi(t)^c f(t) \big|^p \frac{\Psi'(t)}{\Psi(t)} \,\mathrm{d}t \right)^{\frac{1}{p}} \Psi(u)^{c -s - 1} \left(\log \Psi(u)\right)^{\mu - 1} \Psi'(u) \,\mathrm{d}u,
	\end{align*}
	and hence
	\begin{equation*}
	\left\| \prescript{H}{a}I^{\mu,s}_{\Psi(x)} f\right\|_{X_{\Psi, c}^{p}} \leq M \|f\|_{X_{\Psi, c}^{p}}
	\end{equation*}
	where
	\begin{equation} \label{expresionforM}
	M := \frac{1}{\Gamma(\mu)} \int_{\Psi^{-1}(1)}^{\Psi^{-1} \left(\frac{\Psi(b)}{\Psi(a)}\right)}\Psi(u)^{c -s - 1} \left(\log \Psi(u)\right)^{\mu - 1} \Psi'(u) \,\mathrm{d}u.
	\end{equation}
	When $s = c$, then we have
	\begin{equation}
	M = \frac{1}{\Gamma(\mu)} \int_{\Psi^{-1}(1)}^{\Psi^{-1}\left(\frac{\Psi(b)}{\Psi(a)}\right)} \{\log \Psi(u)\}^{\mu - 1} \frac{\Psi'(u)}{\Psi(u)}\,\mathrm{d}u= \frac{1}{\Gamma(\mu + 1)} \left(\log \frac{\Psi(b)}{\Psi(a)} \right)^{\mu}. \nonumber
	\end{equation}
	If $s > c$, then making the substitution $t = (s - c) \log \Psi(u)$ in Eq. \eqref{expresionforM} and by making use of the definition of the incomplete Gamma function, we have
	\begin{equation}
	M=\frac{(s - c)^{-\mu}}{\Gamma(\mu)} \int_{0}^{(s - c)\log\left(\frac{\Psi(b)}{\Psi(a)} \right)} e^{-t} t^{\mu - 1} dt = \frac{(s - c)^{-\mu} }{\Gamma(\mu)}\gamma \left(\mu,  (s - c) \log  \frac{\Psi(b)}{\Psi(a)} \right). \nonumber
	\end{equation}
	Thus the result is proved for $ 1 \leq p < \infty $.

	Now assume that $p = \infty$. Then by using the definition of $\prescript{H}{a}I^{\mu,s}_{\Psi(x)}$ and Eq. \eqref{normforpinfinity}, we have
	\begin{align}
	\left| \Psi(x)^c \prescript{H}{a}I^{\mu,s}_{\Psi(x)} f(x)\right| &= \frac{1}{\Gamma(\mu)} \int_{a}^{x} \left(\frac{\Psi(t)}{\Psi(x)} \right)^{s - c}  \left(\log \frac{\Psi(x)}{\Psi(t)}\right)^{\mu - 1} \big| \Psi(t)^c f(t)\big| \frac{\Psi'(t)}{\Psi(t)}\,\mathrm{d}t \nonumber \\
	&\leq K(x) \|f\|_{X_{\Psi, c}^{\infty}},\label{eq6.7}
	\end{align}
	where the function $K$ is defined by
	\begin{equation*}
	K(x) :=  \frac{1}{\Gamma(\mu)} \int_{a}^{x} \left(\frac{\Psi(t)}{\Psi(x)} \right)^{s - c}  \left(\log \frac{\Psi(x)}{\Psi(t)}\right)^{\mu - 1} \frac{\Psi^{\prime}(t)dt}{\Psi(t)}.
	\end{equation*}
	Substituting $\Psi(u) = \frac{\Psi(x)}{\Psi(t)}$ in the above equation, we get
	\begin{equation*}
	K(x) =  \frac{1}{\Gamma(\mu)} \int_{\Psi^{-1}(1)}^{\Psi^{-1}\left(\frac{\Psi(x)}{\Psi(a)}\right)} \Psi(u)^{c - s}  \left(\log\Psi(u)\right)^{\mu - 1} \frac{\Psi'(u)}{\Psi(u)}\,\mathrm{d}u.
	\end{equation*}
	We note that the integrand is always positive, since $\Psi$ is a positive increasing function and $\Psi(u)\geq1$. So, comparing this integral with the definition \eqref{expresionforM} of $M$ in the previous part of the proof, we see that $K(x)\leq M$, and the result follows for $  p=\infty $.
\end{proof}

\begin{rem}
We can compare Theorem \ref{thm4.1} with the analogous result for the classical Riemann-Liouville fractional integral with respect to a function \cite[Theorem 18.1]{Samko}. In the Riemann--Liouville context, the function space required for $\prescript{RL}{a}I^{\mu}_{\Psi(x)}$ is the same as that required for $\prescript{RL}{a}I^{\mu}_x$. The extra parameter and power function in the Hadamard-type context means that here we should define the new function space $X_{\Psi, c}^{p}(a, b)$.

In the case that $ \Psi(x)=x $, we have an analogous conclusion for the original Hadamard-type fractional integral operator. Considering $ \Psi(x)=x $ and $ s=0 $ gives a corresponding result for the basic Hadamard fractional integral operator. See also \cite{Butzer,3} for these results.

In the case that $\Psi(x)=e^x$, we have an analogous conclusion for the tempered fractional integral operator. But in this case the result is almost trivial, since $\log\frac{\Psi(x)}{\Psi(t)}$ is simply $(x-t)$ and $\frac{\Psi(x)}{\Psi'(x)}$ reduces to $1$.
\end{rem}

\begin{thm} \label{Thm:reviewer}
	Let $ \Psi \in C^1[a,b]$ be a positive increasing function such that $\Psi(b) \leq e$. If $\mu,s>0$ and $p >\max\left(\frac{1}{\mu},1\right)$, then the map $\prescript{H}{a}I^{\mu,s}_{\Psi(x)} : L^{p}[a,b]\rightarrow C[a,b]$ is a compact operator (if we define $\prescript{H}{a}I^{\mu,s}_{\Psi(x)}f(a) := 0$). In particular, if $\mu\in (0,1]$, then $\prescript{H}{a}I^{\mu,s}_{\Psi(x)}:L^{p}[a,b]\rightarrow H^{\mu - \frac{1}{p}}[a,b]$ is a compact operator, were $H^{\mu - \frac{1}{p}}[a,b]$ denotes the H\"{o}lder space of order $\mu - \frac{1}{p} > 0$.
\end{thm}

\begin{proof}
We know $p>1$ has a conjugate exponent $q \in [1, \infty]$. Since $p>\frac{1}{\mu}$, we have $q(\mu -1) > -1$. For any $x_{1},x_{2} \in [a,b]$ with $x_{1} < x_{2}$, an explicit calculation using the H\"{o}lder inequality reveals that
	\begin{align*}
	\left| \Psi(x_2)^s\prescript{H}{a}I^{\mu,s}_{\Psi(x)} f(x_{2}) - \Psi(x_1)^s\prescript{H}{a}I^{\mu,s}_{\Psi(x)} f(x_{1}) \right| \Gamma(\mu) &{\color{white}=}\\ 
	&\hspace{-6cm}= \left| \int_{a}^{x_{2}} {\Psi(t)}^{s} \left(\log \frac{\Psi(x_{2})}{\Psi(t)}\right)^ {\mu - 1} \frac{\Psi '(t)}{\Psi(t)}f(t)\,\mathrm{d}t - \int_{a}^{x_{1}} {\Psi(t)}^{s} \left(\log \frac{\Psi(x_{1})}{\Psi(t)}\right)^ {\mu - 1} \frac{\Psi '(t)}{\Psi(t)}f(t)\,\mathrm{d}t \right| \\
	&\hspace{-6cm}\leq \int_{a}^{x_{1}} \left| {\Psi(t)}^{s} \left(\log \frac{\Psi(x_{2})}{\Psi(t)}\right)^ {\mu - 1}  -  {\Psi(t)}^{s} \left(\log \frac{\Psi(x_{1})}{\Psi(t)}\right)^ {\mu - 1}\right| \left| \frac{\Psi '(t)}{\Psi(t)}f(t) \right|\,\mathrm{d}t \\ & + \int_{x_{1}}^{x_{2}} \left|  {\Psi(t)}^{s} \left(\log \frac{\Psi(x_{2})}{\Psi(t)}\right)^ {\mu - 1} \right| \left| \frac{\Psi '(t)}{\Psi(t)}f(t) \right|\,\mathrm{d}t \\
	&\hspace{-6cm}\leq \|f\|_{L_{p}[a,b]} \left( \int_{a}^{x_{1}} \left|  {\Psi(t)}^{s} \left(\log \frac{\Psi(x_{2})}{\Psi(t)}\right)^ {\mu - 1}  -  {\Psi(t)}^{s} \left(\log \frac{\Psi(x_{1})}{\Psi(t)}\right)^ {\mu - 1}\right|^{q}  \left(\frac{\Psi '(t)}{\Psi(t)}\right)^{q} \,\mathrm{d}t\right)^{\frac{1}{q}} \\ &\hspace{-2cm} + \|f\|_{L_{p}[a,b]} \left( \int_{x_{1}}^{x_{2}} \left|  {\Psi(t)}^{s} \left(\log \frac{\Psi(x_{2})}{\Psi(t)}\right)^ {\mu - 1} \left(\frac{\Psi '(t)}{\Psi(t)}\right)\right|^{q} \,\mathrm{d}t \right)^{\frac{1}{q}} \\
	&\hspace{-6cm}\leq \|\Psi\|_{\infty}^{s}\|f\|_{L_{p}[a,b]} \left(\frac{\|\Psi'\|_{\infty}}{\Psi(a)}\right)^{\frac{1}{p}}\Bigg[\left( \int_{a}^{x_{1}} \left|  \left(\log \frac{\Psi(x_{2})}{\Psi(t)}\right)^ {\mu - 1}  -   \left(\log \frac{\Psi(x_{1})}{\Psi(t)}\right)^ {\mu - 1}\right|^{q}  \frac{\Psi '(t)}{\Psi(t)}\,\mathrm{d}t\right)^{\frac{1}{q}} \\ & +  \left( \int_{x_{1}}^{x_{2}} \left|   \left(\log \frac{\Psi(x_{2})}{\Psi(t)}\right)^ {\mu - 1} \right|^{q} \frac{\Psi '(t)}{\Psi(t)}\,\mathrm{d}t \right)^{\frac{1}{q}} \Bigg].
	\end{align*}
	Thus, since $0 < \log\frac{\Psi(x)}{\Psi(t)} < 1$ for $x \in [a,b]$ and $t \leq x$, and $|x - y|^{\lambda} \leq |x^{\lambda} - y^{\lambda}|$ for $\lambda \geq 1$, it follows that
	\begin{equation*}
	\begin{split}
J_{1} &:= \left( \int_{a}^{x_{1}} \left|  \left(\log \frac{\Psi(x_{2})}{\Psi(t)}\right)^ {\mu - 1}  -   \left(\log \frac{\Psi(x_{1})}{\Psi(t)}\right)^ {\mu - 1}\right|^{q}  \frac{\Psi '(t)}{\Psi(t)}\,\mathrm{d}t\right)^{\frac{1}{q}} \\&\leq \left( \int_{a}^{x_{1}} \left|  \left(\log \frac{\Psi(x_{2})}{\Psi(t)}\right)^ {q(\mu - 1)}  -   \left(\log \frac{\Psi(x_{1})}{\Psi(t)}\right)^ {q(\mu - 1)}\right|  \frac{\Psi '(t)}{\Psi(t)}\,\mathrm{d}t\right)^{\frac{1}{q}} \\&= \frac{1}{\sqrt[q]{q(\mu - 1) + 1}} 
\begin{cases}
 \left(\log \frac{\Psi(x_{2})}{\Psi(x_{1})}\right)^ {\mu - \frac{1}{p}}  &\text{if } \mu < 1, \\
 0   &\text{if } \mu = 1, 
 \\
   \left[   \left(\log \frac{\Psi(x_{2})}{\Psi(a)}\right)^ {q(\mu - 1)+1}  -   \left(\log \frac{\Psi(x_{1})}{\Psi(a)}\right)^ {q(\mu - 1)+1} \right]^{\frac{1}{q}} &\text{if } \mu > 1. 
 \end{cases}
	\end{split}
	\end{equation*}
	Also with some further efforts one can get
	\begin{equation*}
	J_{2}:= \left( \int_{x_{1}}^{x_{2}}    \left(\log \frac{\Psi(x_{2})}{\Psi(t)}\right)^ {q(\mu - 1)} \frac{\Psi '(t)}{\Psi(t)}\,\mathrm{d}t \right)^{\frac{1}{q}} = \frac{1}{\sqrt[q]{q(\mu - 1) + 1}} \left(\log \frac{\Psi(x_{2})}{\Psi(x_{1})}\right)^ {\mu - \frac{1}{p}}.
	\end{equation*}
	A combination of these results yields
	\begin{align} \label{combinationofresults}
	\begin{split}
		\left| \Psi(x_2)^s\prescript{H}{a}I^{\mu,s}_{\Psi(x)} f(x_{2}) - \Psi(x_1)^s\prescript{H}{a}I^{\mu,s}_{\Psi(x)} f(x_{1}) \right| &\leq \frac{ \Psi(b)^{s}\|f\|_{L_{p}[a,b]} \left(\frac{\|\Psi'\|}{\Psi(a)}\right)^{\frac{1}{p}}}{\Gamma(\mu) {\sqrt[q]{q(\mu - 1) + 1}}} \\&\hspace{-5cm} \times\begin{cases}
		2\left(\log \frac{\Psi(x_{2})}{\Psi(x_{1})}\right)^ {\mu - \frac{1}{p}}  &\text{if } \mu < 1, \\
		\left(\log \frac{\Psi(x_{2})}{\Psi(x_{1})}\right)^ {\mu - \frac{1}{p}}    &\text{if } \mu = 1, 
		\\
		\left(\log \frac{\Psi(x_{2})}{\Psi(x_{1})}\right)^ {\mu - \frac{1}{p}}  + \left[   \left(\log \frac{\Psi(x_{2})}{\Psi(a)}\right)^ {q(\mu - 1)+1}  -   \left(\log \frac{\Psi(x_{1})}{\Psi(a)}\right)^ {q(\mu - 1)+1} \right]^{\frac{1}{q}} &\text{if } \mu > 1. 
		\end{cases}
			\end{split}
	\end{align}
	The above estimates yield 
	\begin{equation*}
	\Big| \Psi(x_2)^s\prescript{H}{a}I^{\mu,s}_{\Psi(x)} f(x_{2}) - \Psi(x_1)^s\prescript{H}{a}I^{\mu,s}_{\Psi(x)} f(x_{1}) \Big| \rightarrow 0 \ \  \text{ as } \  x_{2} \rightarrow x_{1}.
	\end{equation*}
	Further, in view of our definition $\prescript{H}{a}{\mathfrak{I}}^{\mu,s}_{\Psi(x)}f(a) := 0$, we observe that
	\begin{align}
	\| \prescript{H}{a}I^{\mu,s}_{\Psi(x)} f\|_{\infty} &= \sup_{x \in [a,b]} \left| \prescript{H}{a}I^{\mu,s}_{\Psi(x)} f(x) - \prescript{H}{a}I^{\mu,s}_{\Psi(x)} f(a) \right| \leq \Psi(a)^{-s}\sup_{x \in [a,b]} \left| \Psi(x)^s\prescript{H}{a}I^{\mu,s}_{\Psi(x)} f(x) - \Psi(a)^s\prescript{H}{a}I^{\mu,s}_{\Psi(x)} f(a) \right| \nonumber \\
	&\leq \Psi(a)^{-s} \frac{2 \Psi(b)^{s}\|f\|_{L_{p}[a,b]} \left(\frac{\|\Psi'\|}{\Psi(a)}\right)^{\frac{1}{p}}}{\Gamma(\mu) {\sqrt[q]{q(\mu - 1) + 1}}} \left(\log\frac{\Psi(b)}{\Psi(a)}\right)^ {\mu - \frac{1}{p}}, \label{norm}
	\end{align}
	this result being valid for any $f \in L^{p}[a,b]$. Therefore, since $\Psi \in C[a,b]$, it follows from \eqref{combinationofresults} that the function $\prescript{H}{a}I^{\mu,s}_{\Psi(x)}f(\cdot)$ is continuous on $ [a,b] $ for any $f\in L^{p}[a,b]$. This points out the fact that the operator $\prescript{H}{a}I^{\mu,s}_{\Psi(x)}$ maps $L^{p}[a,b]$ into $C[a,b]$, given the starting assumption of $p>\max\left(\frac{1}{\mu},1\right)$. In particular, when $\mu \in (0, 1]$, since $\Psi \in C^{1}[a,b] \subset H^{\mu - \frac{1}{p}}[a,b]$, it can be easily shown that $\prescript{H}{a}I^{\mu,s}_{\Psi(x)}$ maps $L^{p}[a,b]$ into the H\"{o}lder space $H^{\mu - \frac{1}{p}}[a,b]$.
	
	Now the continuity of the linear operator $\prescript{H}{a}I^{\mu,s}_{\Psi(x)}$ follows from \eqref{norm}. From \eqref{combinationofresults} the equicontinuity of the image of bounded sets of $L^{p}[a,b]$ under  $\prescript{H}{a}I^{\mu,s}_{\Psi(x)}$ is obvious, and the uniform boundedness of this image follows from \eqref{norm}. By the Arzel\`a--Ascoli theorem, it follows that the map $\prescript{H}{a}I^{\mu,s}_{\Psi(x)} : L^{p}[a,b]\rightarrow C[a,b]$ is compact, which is what we wished to show.
\end{proof}

\subsection{Function spaces for the fractional derivative operator}

In this section, we discuss sufficient conditions for the existence of the Hadamard-type fractional derivative operator $\prescript{HR}{a}D^{\mu,s}_{\Psi(x)}$ with respect to a function. To do this, we define the following function space for $s\in\mathbb{R}$, $a<b$ in $\mathbb{R}$, and $\Psi$ as usual a positive increasing function:
\begin{equation}\label{space}
AC_{\delta^{\Psi}, s}^{n} [a,b] := \left\{ h: [a, b] \rightarrow \mathbb{C} : \left(\frac{\Psi(x)}{\Psi'(x)}\frac{\mathrm{d}}{\mathrm{d}x}\right)^{n-1} \Big[\Psi(x)^s h(x)\Big] \in AC [a, b] \right\}
\end{equation}
where $AC [a, b]$ is the set of absolutely continuous functions on $[a, b]$, which coincides with the space of primitives of Lebesgue measurable functions, i.e. \cite{Samko}:
\begin{equation} \label{primitivecondition}
h(x) \in AC [a, b] \iff \left(h(x) = \int_{a}^{x} \phi(t) dt + c \quad\text{ for some constant $c$ and }\phi \in L^1(a, b)\right).
\end{equation}
In the following theorem, we characterize the space $ AC_{\delta^{\Psi}, s}^{n} [a,b] $.

\begin{thm}\label{ourthm4}
	The space $AC_{\delta^{\Psi}, s}^{n} [a,b]$ consists of those and only those functions $g(x)$ that are represented in the form
	\begin{equation} \label{form}
	g(x) = \Psi(x)^{-s} \left[ \frac{1}{(n - 1)!} \int_{a}^{x} \left(\log \frac{\Psi(x)}{\Psi(t)}\right)^{n - 1} \phi(t) \,\mathrm{d}t + \sum_{k=0}^{n-1} c_k \left(\log \frac{\Psi(x)}{\Psi(a)}\right)^k \right],
	\end{equation}
for some constants $c_0,c_1,\dots,c_{n-1}$ and function $\phi\in L^1(a,b)$.
\end{thm}

\begin{proof} 
Firstly, let $g(x) \in AC_{\delta^{\Psi}, s}^{n} [a,b]$. By the definition \eqref{space} of this function space and the equivalence condition \eqref{primitivecondition}, we have
	\begin{equation} \label{equsedinrem}
	\left(\frac{\Psi(x)}{\Psi'(x)}\frac{\mathrm{d}}{\mathrm{d}x}\right)^{n-1} \Big[\Psi(x)^s g(x)\Big] = \int_a^x \phi(t)\,\mathrm{d}t + c_{n - 1},
	\end{equation}
	i.e.
	\[
	\left(\frac{\Psi(x)}{\Psi'(x)}\frac{\mathrm{d}}{\mathrm{d}x}\right)^{n} \Big[\Psi(x)^s g(x)\Big] = \phi(x)\frac{\Psi(x)}{\Psi'(x)},
	\]
We now invert the $\frac{\Psi(x)}{\Psi'(x)}\frac{\mathrm{d}}{\mathrm{d}x}$ operator, which is differentiation with respect to $\log\Psi(x)$. In order to cancel the $n$th power of this operator on the left-hand side, we need to integrate $n$ times with respect to $\log\Psi(x)$. This means applying the operator
\[
\prescript{RL}{a}I^{n}_{\log\Psi(x)}f(x)=\frac{1}{(n-1)!}\int_a^x\left(\log \frac{\Psi(x)}{\Psi(t)}\right)^{n - 1}f(t)\frac{\Psi'(t)}{\Psi(t)}\,\mathrm{d}t,
\]
and taking into account the extra constant term accrued at each stage of the $n$-fold integration. The result is
\begin{align*}
&\Psi(x)^s g(x) \\
&\hspace{0.5cm}=\prescript{RL}{a}I^{n}_{\log\Psi(x)}\left[\phi(x)\frac{\Psi(x)}{\Psi'(x)}\right]+c_{n-1}\prescript{RL}{a}I^{n-1}_{\log\Psi(x)}(1)+c_{n-2}\prescript{RL}{a}I^{n-2}_{\log\Psi(x)}(1)+\dots+c_1\prescript{RL}{a}I^{1}_{\log\Psi(x)}(1)+c_0 \\
&\hspace{0.5cm}=\frac{1}{(n-1)!}\int_a^x\left(\log \frac{\Psi(x)}{\Psi(t)}\right)^{n - 1}\phi(t)\,\mathrm{d}t+\sum_{k=0}^{n-1}\frac{c_k}{k!}\left(\log\frac{\Psi(x)}{\Psi(a)}\right)^k.
\end{align*}
Replacing $c_k$ by $k!c_k$ in the notation, we get Eq. \eqref{form}.

Conversely, let $g(x)$ be represented in the form Eq. \eqref{form}, i.e.
	\begin{equation*}
	\Psi(x)^s g(x) =  \frac{1}{(n - 1)!} \int_{a}^{x} \left(\log \frac{\Psi(x)}{\Psi(t)}\right)^{n - 1} \phi(t)\,\mathrm{d}t + \sum_{k=0}^{n-1} c_k \left(\log \frac{\Psi(x)}{\Psi(a)}\right)^k.
	\end{equation*}
Differentiating $n-1$ times with respect to $\log\Psi(x)$, we get
	\begin{align*}
	\left(\frac{\Psi(x)}{\Psi'(x)}\frac{\mathrm{d}}{\mathrm{d}x}\right)^{n-1}\Big[\Psi(x)^s g(x)\Big] &= \int_{a}^{x}\phi(t) \,\mathrm{d}t + \sum_{k=0}^{n - 1}\frac{k! c_k}{(k - n+1)!} \left(\log \frac{\Psi(x)}{\Psi(a)}\right)^{k - n+1} \\
	&=\int_a^x\phi(t) \,\mathrm{d}t + (n-1)! c_{n-1}.
	\end{align*}
Hence, in accordance with Eq. \eqref{primitivecondition}, we deduce that $g(x) \in AC_{\delta^{\Psi}, s}^{n} [a,b]$, and the proof is complete.
\end{proof}

\begin{rem}
It can be seen from our proof above that $c_{k} = \frac{g_{k}(a)}{k!}$ for all $k = 0, 1, ..., n - 1$, where
\begin{equation}
\label{gk}
g_{k}(x) := \left(\frac{\Psi(x)}{\Psi'(x)}\frac{\mathrm{d}}{\mathrm{d}x}\right)^k\Big[\Psi(x)^s g(x)\Big]=\Psi(x)^s\cdot\prescript{H}{}D^k_{\Psi(x)}g(x),
\end{equation}
and also $\phi(x)=g_{n-1}(x)$. Hence, as a more specific statement of the result of Theorem \ref{ourthm4}, the function $g(x)$ can be represented as
	\begin{equation} \label{anotherform}
	g(x) = \Psi(x)^{-s} \left[ \frac{1}{(n-1)!}\int_{a}^{x} \left(\log \frac{\Psi(x)}{\Psi(t)}\right)^{n - 1} g'_{n-1}(t) \,\mathrm{d}t + \sum_{k=0}^{n - 1}\frac{g_{k}(a)}{k!} \left(\log \frac{\Psi(x)}{\Psi(a)}\right)^{k} \right].
	\end{equation}
	
\end{rem}

Now we prove a result giving sufficient conditions for the existence of the Hadamard-type fractional derivative of a function with respect to another function.

\begin{thm}\label{ourthm5}
	Let $\mu > 0$, $n = \lfloor\mu\rfloor + 1$, $s \in \R$, $\Psi$ an increasing positive function on $[a,b]\subset\mathbb{R}$, and $g\in AC_{\delta^{\Psi}, s}^{n} [a,b]$. Then the Hadamard-type fractional derivative $\prescript{HR}{a}D^{\mu,s}_{\Psi(x)}g(x)$ of $g$ with respect to $\Psi$ exists almost everywhere on $[a, b]$, and it may be represented in the form
	\begin{equation}
	\prescript{HR}{a}D^{\mu,s}_{\Psi(x)}g(x) = \Psi(x)^{-s} \left[  \frac{1}{\Gamma(n-\mu)}\int_{a}^{x} \left(\log \frac{\Psi(x)}{\Psi(t)}\right)^{n - \mu - 1} g'_{n-1}(t)\,\mathrm{d}t + \sum_{k=0}^{n - 1}\frac{g_{k}(a)}{\Gamma(k - \mu+ 1)} \left(\log \frac{\Psi(x)}{\Psi(a)}\right)^{k - \mu}\right],
	\end{equation}
	where the functions $g_k$ are defined by \eqref{gk}.
\end{thm}

\begin{proof}
We have $g(x) \in AC_{\delta^{\Psi}, s}^{n} [a,b]$, so we can use the representation for $g(x)$ given in Eq. \eqref{anotherform}. Substituting this into the definition of $\prescript{HR}{a}D^{\mu,s}_{\Psi(x)}g(x)$, we get
\begin{align*}
\prescript{HR}{a}D^{\mu,s}_{\Psi(x)}g(x) &= \Psi(x)^{-s} D^n_{\log\Psi(x)}\prescript{}{a}I^{n-\mu}_{\log\Psi(x)}\Big[\Psi(x)^sg(x)\Big] \\
&\hspace{-1cm}=\Psi(x)^{-s} D^n_{\log\Psi(x)}\prescript{}{a}I^{n-\mu}_{\log\Psi(x)}\left[ \frac{1}{(n-1)!}\int_{a}^{x} \left(\log \frac{\Psi(x)}{\Psi(t)}\right)^{n - 1} g'_{n-1}(t) \,\mathrm{d}t + \sum_{k=0}^{n - 1}\frac{g_{k}(a)}{k!} \left(\log \frac{\Psi(x)}{\Psi(a)}\right)^{k} \right] \\
&\hspace{-1cm}=\Psi(x)^{-s} D^n_{\log\Psi(x)}\prescript{}{a}I^{n-\mu}_{\log\Psi(x)}\left[ \prescript{}{a}I^{n}_{\log\Psi(x)}g'_{n-1}(x)+ \sum_{k=0}^{n - 1}\frac{g_{k}(a)}{k!} \left(\log \frac{\Psi(x)}{\Psi(a)}\right)^{k} \right] \\
&\hspace{-1cm}=\Psi(x)^{-s} \prescript{}{a}I^{n-\mu}_{\log\Psi(x)}g'_{n-1}(x)+ \sum_{k=0}^{n - 1}\frac{g_{k}(a)}{k!}\cdot\frac{k!}{\Gamma(k-\mu+1)} \left(\log \frac{\Psi(x)}{\Psi(a)}\right)^{k-\mu},
\end{align*}
where we have used the results of both Corollary \ref{Cor:ourconjug} and Proposition \ref{Prop:ourlem3} respectively. This is the required result.
\end{proof}

\begin{rem}
	The result of Theorem \ref{ourthm5} is analogous to the results for the classical Riemann--Liouville fractional differential operator which may be found in \cite{Samko}. Indeed, the proof also follows the same structure, differentiating and integrating with respect to $\log\Psi(x)$ instead of with respect to $x$, and with the extra $\Psi(x)^s$ multiplier.
	
	In the case that $ \Psi(x)=x $, we have an analogous conclusion for the original Hadamard-type fractional differential operator -- and, in the case $ \Psi(x)=x $ and $ s=0 $, for the Hadamard fractional differential operator. See also \cite{3} for these results.
	
	In the case that $\Psi(x)=e^x$, we have an analogous conclusion for the tempered fractional differential operator.
\end{rem}

\section{Integration by parts in Hadamard-type fractional calculus with respect to functions} \label{Sec:intparts}

One of the most important techniques of classical calculus is integration by parts. Versions of this in fractional calculus have been a subject of study for a long time \cite{love-young}, and in this
section we derive some integration by parts formulae in the setting of Hadamard-type fractional calculus with respect to functions.

First, it is necessary to define the right-sided versions of these new fractional operators. This is done in a natural way following Definition \ref{defn:3}, as follows.

    \begin{defn}\label{Def:rightsided}
     Let $\mu\in\mathbb{R}$ with $\mu>0$ (or $\mu\in\mathbb{C}$ with $\mathrm{Re}(\mu)>0$), and $s\in\mathbb{C}$. Let $f$ be an integrable function defined on $[a,b]$ where $a<b$ in $\mathbb{R}$, and $\Psi\in C^1([a,b])$ be a positive increasing function such that $ \Psi'(x) \neq 0 $ for all $ x \in [a,b] $. Then, the right-sided Hadamard-type fractional integral of $f$ with respect to $\Psi$, or the right-sided tempered fractional integral of $f$ with respect to $\log\circ\Psi$, with order $\mu$ and parameter $s$, is defined as
	\begin{equation}
	\prescript{H}{\Psi(x)}I^{\mu,s}_{b} f(x) = \prescript{T}{\log\Psi(x)}I^{\mu,s}_{b} f(x) = \frac{1}{\Gamma(\mu)} \int_{x}^{b} \left(\frac{\Psi(x)}{\Psi(t)}\right)^{s} \left(\log \frac{\Psi(t)}{\Psi(x)}\right)^ {\mu - 1} f(t) \frac{\Psi '(t)}{\Psi(t)}\,\mathrm{d}t,\quad\quad x\in[a,b].
	\end{equation}

	Writing $n=\lfloor\mu\rfloor+1$ (or $n=\lfloor\mathrm{Re}(\mu)\rfloor+1$) so that $ n-1 \leq \mu < n\in\mathbb{Z}^+$, the right-sided Hadamard-type fractional derivative of $f$ with respect to $\Psi$, or the right-sided tempered fractional derivative of $f$ with respect to $\log\circ\Psi$, with order $\mu$ and parameter $s$, is defined as (in the Riemann--Liouville sense)
	\begin{equation}
	\prescript{HR}{\Psi(x)}D^{\mu,s}_{b} f(x) = \prescript{TR}{\log\Psi(x)}D^{\mu,s}_{b} f(x) = \prescript{H}{\Psi(x)}D^{n,s}_{b} \prescript{H}{\Psi(x)}I^{n - \mu,s}_{b} f(x), \quad\quad x\in[a,b],
	\end{equation}
or (in the Caputo sense)
	\begin{equation}
	\prescript{HC}{\Psi(x)}D^{\mu,s}_{b} f(x) = \prescript{TC}{\log\Psi(x)}D^{\mu,s}_{b} f(x) = \prescript{H}{\Psi(x)}I^{n - \mu,s}_{b} \prescript{H}{\Psi(x)}D^{n,s}_{b} f(x), \quad\quad x\in[a,b],
	\end{equation}
where the $n$th-order derivative is defined by
	\begin{equation}
	\prescript{H}{\Psi(x)}D^{n,s}_{b}f(x) = (-1)^n \Psi(x)^{s} \left(\frac{\Psi(x)}{\Psi'(x)}\cdot\frac{\mathrm{d}}{\mathrm{d}x}\right)^{n} \Big[\Psi(x)^{-s}f(x)\Big]. 
	\end{equation}
\end{defn}

\begin{lem}\label{lemmaintbyparts} Let $ \mu>0 $, $n=\lfloor\mu\rfloor+1$, $s \in \R$, $\Psi$ an increasing positive function on $[a,b]\subset\mathbb{R}$, and $ p \geq 1 $, $ q \geq 1 $ with $ \frac{1}{p} + \frac{1}{q} \leq 1+\mu $  ( assuming $p\neq1$ and $q\neq1$ in the case when $ \frac{1}{p} + \frac{1}{q}=1+\mu $). If $ f\in X^{p}_{\Psi,c}(a,b) $ and $ g\in X^{q}_{\Psi,c}(a,b) $, then
	\begin{equation*}
	\int_{a}^{b} \frac{\Psi'(x)}{\Psi(x)}f(x) \prescript{H}{a}I^{\mu,s}_{\Psi(x)}g(x)\,\mathrm{d}x = \int_{a}^{b} \frac{\Psi'(x)}{\Psi(x)}g(x) 	\prescript{H}{\Psi(x)}I^{\mu,s}_{b}f(x)\,\mathrm{d}x.
	\end{equation*}
\end{lem}

\begin{proof}
Using the definitions of the Hadamard-type fractional integrals with respect to $\Psi(x)$, and using Fubini's theorem for swapping integrals, we have
	\begin{align*}
	\int_{a}^{b} \frac{\Psi'(x)}{\Psi(x)}f(x) \prescript{H}{a}I^{\mu,s}_{\Psi(x)}g(x)\,\mathrm{d}x &= \frac{1}{\Gamma(\mu)}\int_{a}^{b} \frac{\Psi^{\prime}(x)}{\Psi(x)}f(x) \int_{a}^x \left(\frac{\Psi(t)}{\Psi(x)}\right)^s \left(\log \frac{\Psi(x)}{\Psi(t)}\right)^ {\mu - 1} g(t) \frac{\Psi '(t)}{\Psi(t)}\,\mathrm{d}t\,\mathrm{d}x \\
	&=\frac{1}{\Gamma(\mu)}\int_{a}^{b} \frac{\Psi^{\prime}(t)}{\Psi(t)}g(t) \int_{t}^b \left(\frac{\Psi(t)}{\Psi(x)}\right)^s \left(\log \frac{\Psi(x)}{\Psi(t)}\right)^ {\mu - 1} f(x) \frac{\Psi '(x)}{\Psi(x)}\,\mathrm{d}x\,\mathrm{d}t \\
	&= \int_{a}^{b} \frac{\Psi'(t)}{\Psi(t)}g(t) \Big[\prescript{H}{\Psi(t)}I^{\mu,s}_{b}f(t)\Big]\,\mathrm{d}t.
	\end{align*}
	Hence, the result is proved.
\end{proof}

\begin{thm}
Let $\mu > 0$, $n = \lfloor\mu\rfloor + 1$, $s \in \R$, $\Psi$ an increasing positive function on $[a,b]\subset\mathbb{R}$, and $1\leq p\leq\infty$. If $ f\in AC^{n}_{\delta^{\Psi},s} [a,b] $ and $ g\in X^{p}_{\Psi,c}(a,b) $, then we have the following integration by parts relations for Hadamard-type fractional integrals with respect to $\Psi(x)$:
	\begin{multline*}
	\int_{a}^{b} f(x) \prescript{HR}{a}D^{\mu,s}_{\Psi(x)} g(x)\,\mathrm{d}x = \int_{a}^{b} \frac{\Psi'(x)}{\Psi(x)} g(x)\prescript{HC}{\Psi(x)}D^{\mu,s}_{b}\left(\frac{\Psi(x)}{\Psi'(x)} f(x)\right)\,\mathrm{d}x \\ + \sum_{k=0}^{n-1} \left[\prescript{H}{\Psi(x)}D^{k,s}_{b} \left( \frac{\Psi(x)}{\Psi'(x)} f(x) \right) \prescript{H}{a}I^{k-\mu+1,s}_{\Psi(x)} g(x) \right]^{b}_{a},
	\end{multline*}
and
	\begin{multline*}
	\int_{a}^{b} f(x) \prescript{H}{\Psi(x)}D^{\mu,s}_{b} g(x)\,\mathrm{d}x = \int_{a}^{b} \frac{\Psi'(x)}{\Psi(x)} g(x)\prescript{HC}{a}D^{\mu,s}_{\Psi(x)}\left(\frac{\Psi(x)}{\Psi'(x)} f(x) \right)\,\mathrm{d}x \\ - \sum_{k=0}^{n-1}\left[\prescript{H}{a}D^{k,s}_{\Psi(x)} \left( \frac{\Psi(x)}{\Psi'(x)} f(x) \right) \prescript{H}{\Psi(x)}I^{k-\mu+1,s}_{b} g(x) \right]^{b}_{a}.
	\end{multline*}
\end{thm}

\begin{proof}
We prove only the first of the two stated identities, since the proofs of both are nearly identical to each other.

Using the definition of $ \prescript{HR}{a}D^{\mu,s}_{\Psi(x)}$, we have
	\begin{align*}
	\int_{a}^{b} f(x) \prescript{HR}{a}D^{\mu,s}_{\Psi(x)} g(x) \,\mathrm{d}x &= \int_{a}^{b} f(x) \prescript{H}{a}D^{n,s}_{\Psi(x)} \prescript{HR}{a}D^{n-\mu,s}_{\Psi(x)} g(x) \,\mathrm{d}x \\
	&= \int_{a}^{b} f(x) \Psi(x)^{-s} \frac{\Psi(x)}{\Psi'(x)} \frac{\mathrm{d}}{\mathrm{d}x} \left[ \left( \frac{\Psi(x)}{\Psi'(x)} \frac{\mathrm{d}}{\mathrm{d}x} \right)^{n-1} \Psi(x)^s \prescript{H}{a}I^{n-\mu,s}_{\Psi(x)} g(x)\right] \mathrm{d}x.
	\end{align*}
	Using integration by parts, we find
	\begin{align*}
	\int_{a}^{b} f(x) \prescript{HR}{a}D^{\mu,s}_{\Psi(x)} g(x) \,\mathrm{d}x &= \left[f(x) \Psi(x)^{-s} \frac{\Psi(x)}{\Psi'(x)} \left( \frac{\Psi(x)}{\Psi'(x)} \frac{\mathrm{d}}{\mathrm{d}x} \right)^{n-1}	 \Psi(x)^s \prescript{H}{a}I^{n-\mu,s}_{\Psi(x)} g(x)\right]^{b}_{a} \\
	& \hspace{0.5cm}- \int_{a}^{b} \left[ \left( \frac{\Psi(x)}{\Psi'(x)} \frac{\mathrm{d}}{\mathrm{d}x} \right)^{n-1} \Psi(x)^s \prescript{H}{a}I^{n-\mu,s}_{\Psi(x)} g(x) \right]\frac{\mathrm{d}}{\mathrm{d}x} \left( f(x) \Psi(x)^{-s} \frac{\Psi(x)}{\Psi'(x)}\right) \mathrm{d}x \\
	&= \left[ \frac{\Psi(x)}{\Psi'(x)} f(x) \prescript{H}{a}D^{n-1,s}_{\Psi(x)} \prescript{H}{a}I^{n-\mu,s}_{\Psi(x)} g(x) \right]^{b}_{a} \\
	& \hspace{1.5cm}+ \int_{a}^{b} \frac{\Psi'(x)}{\Psi(x)} \prescript{H}{a}D^{n-1,s}_{\Psi(x)} \prescript{H}{a}I^{n-\mu,s}_{\Psi(x)} g(x) \prescript{H}{\Psi(x)}D^{1,s}_{b}\left[ \frac{\Psi(x)}{\Psi'(x)} f(x) \right] \mathrm{d}x.
	\end{align*}
	Again applying integration by parts:
	\begin{multline*}
	\int_{a}^{b} f(x) \prescript{HR}{a}D^{\mu,s}_{\Psi(x)} g(x) \,\mathrm{d}x = \left[ \frac{\Psi(x)}{\Psi'(x)} f(x) \prescript{H}{a}D^{n-1,s}_{\Psi(x)} \prescript{H}{a}I^{n-\mu,s}_{\Psi(x)} g(x) \right]^{b}_{a}  \\ + \left[\prescript{H}{\Psi(x)}D^{1,s}_{b}\left(\frac{\Psi(x)}{\Psi'(x)} f(x)\right) \prescript{H}{a}D^{n-2,s}_{\Psi(x)} \prescript{H}{a}I^{n-\mu,s}_{\Psi(x)} g(x) \right]^{b}_{a} \\ + \int_{a}^{b} \frac{\Psi^{\prime}(x)}{\Psi(x)} \prescript{H}{a}D^{n-2,s}_{\Psi(x)} \prescript{H}{a}I^{n-\mu,s}_{\Psi(x)} g(x) \prescript{H}{\Psi(x)}D^{2,s}_{b}\left( \frac{\Psi(x)}{\Psi'(x)} f(x)  \right) \,\mathrm{d}x.
	\end{multline*}
	Continuing in this manner, applying integration by parts $n$ times, we get in the end
	\begin{multline*}
	\int_{a}^{b} f(x) \prescript{HR}{a}D^{\mu,s}_{\Psi(x)} g(x) \,\mathrm{d}x = \sum_{k=0}^{n-1} \left[\prescript{H}{\Psi(x)}D^{k,s}_{b} \left[ \frac{\Psi(x)}{\Psi'(x)} f(x) \right] \prescript{H}{a}D^{n-k-1,s}_{\Psi(x)} \prescript{H}{a}I^{n-\mu,s}_{\Psi(x)} g(x) \right]^{b}_{a} \\ 
	+ \int_{a}^{b} \frac{\Psi'(x)}{\Psi(x)} \prescript{H}{a}I^{n-\mu,s}_{\Psi(x)} g(x) \prescript{H}{\Psi(x)}D^{n,s}_{b} \left[ \frac{\Psi(x)}{\Psi'(x)} f(x) \right] \,\mathrm{d}x.
	\end{multline*}
	Using Lemma \ref{lemmaintbyparts} on the integral from the right-hand side, and recalling Proposition \ref{Prop:ourlem2}, we have
	\begin{multline*}
	\int_{a}^{b} f(x) \prescript{HR}{a}D^{\mu,s}_{\Psi(x)} g(x) \,\mathrm{d}x = \sum_{k=0}^{n-1} \left[\prescript{H}{\Psi(x)}D^{k,s}_{b} \left[ \frac{\Psi(x)}{\Psi'(x)} f(x) \right] \prescript{H}{a}I^{k-\mu+1,s}_{\Psi(x)} g(x) \right]^{b}_{a} \\ + \int_{a}^{b} \frac{\Psi'(x)}{\Psi(x)} g(x) \prescript{H}{\Psi(x)}I^{n-\mu,s}_{b} \prescript{H}{\Psi(x)}D^{n,s}_{b} \left[ \frac{\Psi(x)}{\Psi'(x)} f(x)  \right] \,\mathrm{d}x.
	\end{multline*}
	Finally, by using the definition of $\prescript{HC}{\Psi(x)}D^{\mu,s}_{b}$, we get the required result.
\end{proof}

\section{Conclusion} \label{Sec:concl}

We have demonstrated in this paper that Hadamard-type fractional calculus with respect to the independent variable $x$ is precisely the same as tempered fractional calculus with respect to the logarithm function $\log x$. This new connection will be valuable in the study of both Hadamard-type and tempered fractional calculus, as any results known in one model can now easily be transferred to the other model.

In light of this new connection, it is natural to introduce a more general operator which covers both tempered and Hadamard-type as special cases, namely by applying these operators with respect to arbitrary monotonic functions. The new operator can generalise many existing fractional operators, including Riemann--Liouville, Caputo, classical Hadamard, Hadamard-type, and tempered fractional calculus, plus all of these with respect to arbitrary monotonic functions.

In this paper, we have established several important properties for the Hadamard-type and tempered operators with respect to functions. Semigroup and reciprocal properties have been proved, demonstrating that these operators work together with each other in a natural way. Function spaces have been defined in which the Hadamard-type fractional integral with respect to $\Psi$ is a bounded operator, and sufficient conditions for the existence of the Hadamard-type fractional derivative with respect to $\Psi$ have been established. We have also derived fractional integration by parts formulae in the settings of these operators.

In future work, we plan to introduce an integral transform which will help to solve fractional differential equations in the setting of the Hadamard-type fractional operators with respect to functions. It will also be possible to extend various results in the literature on Hadamard-type fractional calculus, including very recent work such as \cite{abdalla-salem-cichon}, into the new generalised context of Hadamard-type operators with respect to functions.

\section*{Acknowledgements}

The authors would like to thank the editor and the anonymous reviewers for their helpful suggestions. In particular, we are grateful to one of the reviewers for suggesting to us the result and proof that became our Theorem \ref{Thm:reviewer}.

\end{document}